\documentclass[11pt]{amsart}
\usepackage{amsfonts,amssymb,amsthm}
\usepackage{booktabs}
\usepackage[numbers,sort&compress]{natbib}
\usepackage[leqno]{amsmath}
\usepackage[mathscr]{euscript}
\usepackage{graphicx} 
\usepackage{subcaption}
\usepackage{color}
\usepackage{bm}
\usepackage{cases}
\usepackage{mathrsfs}
\usepackage{paralist}
\usepackage{hyperref}
\usepackage{framed}
\usepackage{epstopdf}
\usepackage[ruled, vlined]{algorithm2e}
\usepackage{caption}
\usepackage{tablefootnote}
\newtheorem {theorem}{Theorem}[section]
\newtheorem {proposition}{Proposition}[section]

\newtheorem {lemma}{Lemma}[section]
\newtheorem {example}{Example}[section]
\newtheorem {definition}{Definition}[section]
\newtheorem {remark}{Remark}[section]
\newtheorem {assumption}{Assumption}[section]

\newcommand{\M}{{\rm \bf M}}
\newcommand{\R}{{\mathbb R}}
\newcommand{\N}{\mathbb{N}}
\newcommand{\ba}{{\boldsymbol \alpha}}
\newcommand{\bb}{{\boldsymbol \beta}}
\newcommand{\bx}{\textrm{\bf{x}}}
\newcommand{\bz}{\textrm{\bf{z}}}
\newcommand{\by}{\textrm{\bf{y}}}
\newcommand{\va}{\textrm{\bf{a}}}
\newcommand{\vb}{\textrm{\bf{b}}}
\newcommand{\be}{\textrm{\bf{e}}}
\newcommand{\bze}{\boldsymbol{0}}
\newcommand{\p}{\textrm{\bf{p}}}
\newcommand{\q}{\textrm{\bf{q}}}
\newcommand{\bfd}{\textrm{\bf{d}}}
\newcommand{\dd}{d_{\max}}

\makeatletter
\newcommand{\leqnomode}{\tagsleft@true}
\newcommand{\reqnomode}{\tagsleft@false}
\makeatother

\def\ees{{\accent"5E e}\kern-.385em\raise.2ex\hbox{\char'23}\kern-.08em}
\def\EES{{\accent"5E E}\kern-.5em\raise.8ex\hbox{\char'23 }}
\def\ow{o\kern-.42em\raise.82ex\hbox{
   \vrule width .12em height .0ex depth .075ex \kern-0.16em \char'56}\kern-.07em}
\def\OW{O\kern-.460em\raise1.36ex\hbox{
\vrule width .13em height .0ex depth .075ex \kern-0.16em \char'56}\kern-.07em}

\title{SPLD polynomial optimization and bounded degree SOS hierarchies}

\author{Liguo Jiao}
\address[Liguo Jiao]{Academy for Advanced Interdisciplinary Studies, Northeast Normal University, Changchun 130024, Jilin Province, China; Shanghai Zhangjiang Institute of Mathematics, Shanghai 201203, China.}
\email{jiaolg356@nenu.edu.cn; hanchezi@163.com}

\author{Jae Hyoung Lee$^{\dag}$}
\address[Jae Hyoung Lee]{Department of Applied Mathematics, Pukyong National University, Busan 48513, Korea.}
\email{mc7558@pknu.ac.kr; mc7558@naver.com}

\author{Nguyen Bui Nguyen Thao}
\address[Nguyen Bui Nguyen Thao]{Department of Pedagogy, Dalat University, Vietnam.}
\email{nguyenbnt@dlu.edu.vn}

\thanks{$^{\dag}$Corresponding Author}



\date{\today}
\begin{document}

\begin{abstract}
In this paper,  we introduce a new class of structured polynomials, called  {\it separable plus lower degree {\rm (SPLD)} polynomials}.
The formal definition of an SPLD polynomial, which extends the concept of SPQ polynomials (Ahmadi et al. in Math Oper Res 48:1316--1343, 2023), is provided.
A type of bounded degree SOS hierarchy, referred to as BSOS-SPLD, is proposed to efficiently solve optimization problems involving SPLD polynomials.
Numerical experiments on several benchmark problems indicate that the proposed method yields better performance than the standard bounded degree SOS hierarchy (Lasserre et al. in EURO J Comput Optim 5:87--117, 2017).
An exact SOS relaxation for a class of convex SPLD polynomial optimization problems is proposed.
Finally, we present an application of SPLD polynomials to convex polynomial regression problems arising in statistics.
\end{abstract}

\maketitle



\section{Introduction}

In his seminal monograph \cite{Lasserre2015book}, Lasserre emphasizes that ``if symmetries or some structured sparsity in the problem data are detected and taken into account, then (polynomial) problems of much larger size can be solved."
In this context, we concentrate on polynomial optimization problems.
A standard polynomial optimization problem is of the form:
\begin{align}\label{POP}
\min\limits_{\bx\in\R^n} \ \left\{ f(\bx) \colon g_i(\bx)\geq 0, \ i=1,\ldots,m \right\}, \tag{${\rm POP}$}
\end{align}
where $f,g_i\colon\R^n\to\R,$ $i=1,\ldots,m,$ are real multivariate polynomials.
The standard Lasserre hierarchy, which is a sum-of-squares (SOS) based hierarchy of semidefinite programming (SDP) relaxations, has been demonstrated to be a powerful and effective approach to {\it globally} solving polynomial optimization problems~\eqref{POP}; see, e.g., \cite{Lasserre2001,Lasserre2010,Parrilo2003}.
The convergence of the Lasserre hierarchy (generically finite \cite{Nie2013,Nie2014}) relies on a fundamental SOS representation of a positive polynomial introduced by Putinar~\cite{Putinar1993}.

It is widely acknowledged that the size of the semidefinite programs in the Lasserre hierarchy grows combinatorially with the number of variables and the degree of the relaxation, which poses significant computational challenges for high-dimensional or high-degree polynomial optimization problems. Due to the current limitations of semidefinite solvers, its practical applicability is often restricted to problems of modest size, unless additional structures such as sparsity or symmetry are exploited; see, e.g., \cite{Lasserre2015book}.
To overcome this limitation, several enhancements to the Lasserre hierarchy have been proposed, targeting various aspects of polynomial optimization, as outlined below.

\begin{itemize}
\item[(i)] One direct approach is to design customized algorithms, such as those in \cite{Ghaddar2016,NieWang2012}, that avoid using interior-point methods when solving the SDP problems of the hierarchy.
\item[(ii)] By taking special structures (such as symmetries or certain sparsity patterns of the polynomial optimization problem) into account and restricting the sum-of-squares polynomials in the hierarchy accordingly, one can solve problems of much larger size; see, e.g., the correlative sparsity in \cite{Waki2006,Weisser2018} and the more recent hierarchies based on term sparsity and chordal approaches in \cite{Wang2021,Wang2021-1}.
It is worth noting that the Lasserre hierarchy has recently been shown to solve industrial-scale optimal power flow problems in electrical engineering with several thousand variables
and constraints using restricted sum-of-squares in the hierarchy; see, e.g., \cite{Josz2018,Wang2023book}.
\item[(iii)] Another approach is to exploit alternative conic programming hierarchies with lower computational cost, such as the scaled diagonally dominant sum of squares hierarchies \cite{AhmadiHall2019,AhmadiMajumdar2019,Kuang2019}.
\item[(iv)] (Bounded degree SOS hierarchy \cite{Lasserre2017} and its sparse version \cite{Weisser2018}) The main idea is to restrict the degree of the sum-of-squares of the hierarchy by using Krivine--Stengle's certificate of positivity (rather than Putinar) in real algebraic geometry (see, e.g., \cite{Krivine1964,Stengle1974}), and fix in advance the size of the semidefinite matrix for each relaxation in the hierarchy (in contrast to the standard Lasserre hierarchy).
\end{itemize}

\subsection*{Contributions}
In this paper, we study a subclass of polynomial optimization problems by combining items (ii) and (iv) mentioned above.
Overall, this paper makes the following contributions to polynomial optimization.
\begin{itemize}
\item[{\rm (C1)}] We introduce and study a new class of structured polynomials called SPLD ({\bf S}eparable {\bf P}lus {\bf L}ower {\bf D}egree) polynomials.
We present the formal definition of an SPLD polynomial below.
\begin{definition}[SPLD polynomial]\label{SPLD_poly}{\rm
A nonconstant polynomial $f: \R^n \to \R$ is said to be {\it separable plus lower degree} (SPLD) if
\begin{align*}
f(\bx) = s(\bx) + l(\bx)
\end{align*}
for all $\bx\in\R^n,$ where $s$ is a separable polynomial, i.e., $s(\bx):= \sum_{j = 1}^ns_j(x_j)$ for some univariate polynomials $s_j: \R \to \R,$ and $l$ is a polynomial with its degree lower than the degree of $s.$
}\end{definition}
\begin{remark}[{\bf Why SPLD polynomial?}]{\rm
The rationale for introducing SPLD polynomials stems primarily from the following considerations.
\begin{enumerate}[\upshape (a)]
\item If $l(\bx)$ in Definition~\ref{SPLD_poly} is a quadratic function, the SPLD polynomial coincides with the class of separable plus quadratic (SPQ) polynomials introduced by Ahmadi et al. \cite{Ahmadi2022}.
In this sense, SPLD polynomials include SPQ polynomials and some well-known test functions (e.g., the six-hump camelback function) as special cases.
\item For convex polynomials, as shown in Theorem~\ref{thm3} (see also Remark~\ref{rmk1}), it is reasonable to expect that convex polynomials can often be expressed as SPLD polynomials.
\item An SPLD polynomial is typically a {\em sparse} polynomial\footnote{A sparse polynomial is a polynomial with relatively few nonzero terms compared to the total number of possible monomials of a given degree and number of variables.} in the usual term sparsity sense: the high-degree part contributes only univariate monomials, and all multivariate interactions are confined to the lower-degree term $l$.
\item As a final note, from a practical viewpoint, the SPLD structure may be computationally attractive in situations where the objective and constraints consist of high-degree separable terms together with a comparatively low-degree multivariate term (see, e.g., the portfolio model in Section~\ref{subsec:4} and the SOS-convex SPLD polynomial regression in Section~\ref{subsec:5}).
\end{enumerate}
}\end{remark}
\item[{\rm (C2)}] Even with a small number of variables, the degree of an SPLD polynomial may be quite large because the separable term $s$ can contain high-degree terms.
In such cases, the bounded degree SOS hierarchy (BSOS) introduced by Lasserre et al.~\cite{Lasserre2017} may not be practically implementable due to the size of the associated Gram matrix.
To address this, we propose a variant of the bounded degree SOS hierarchy, called BSOS-SPLD, for optimization problems with SPLD polynomials over a compact semi-algebraic feasible set.
This hierarchy makes use of relaxations associated with Krivine--Stengle's certificate of positivity in real algebraic geometry \cite{Krivine1964,Stengle1974}.
We show that, under mild assumptions, the optimal values of both the primal and dual BSOS-SPLD relaxations converge to the global optimal value of the original polynomial optimization problem.
We also give a sufficient condition for finite convergence of the BSOS-SPLD hierarchy and for extracting a global minimizer.

Compared with the work of Ahmadi et al.~\cite{Ahmadi2022} on convex SPQ polynomials, which provides semidefinite programming formulations for convex SPQ optimization problems, our approach deals with a broader SPLD class (possibly nonconvex and of higher degree) and develops BSOS-type hierarchies based on Krivine--Stengle positivity certificates for global optimization over compact semi-algebraic sets.
\item[{\rm (C3)}] Since SPLD polynomials are typically sparse in their monomial support, one may also consider term-sparsity–exploiting moment-SOS hierarchies, such as TSSOS and Chordal-TSSOS \cite{Wang2021,Wang2021-1}.
In this paper, we take a different approach: we adapt the BSOS hierarchy based on the Krivine--Stengle positivity certificate to SPLD polynomials, thereby directly exploiting the algebraic decomposition $f=s+l$.

Moreover, even when an SPLD polynomial optimization problem is dense in the sense of correlative sparsity (e.g., when the objective/constraints involve many variables simultaneously), decompositions based on correlative sparsity may yield limited reduction.
Nevertheless, BSOS-SPLD may still exploit the algebraic decomposition $f=s+l$ to form smaller SDP blocks determined primarily by the lower-degree term, thus reducing the computational cost of the resulting semidefinite relaxations.

\item[{\rm (C4)}] An exact SOS relaxation for a class of SOS-convex SPLD polynomial optimization problems is proposed.
As a practical application, we consider polynomial regression problems in statistics using SOS-convex SPLD polynomials.
\end{itemize}

\subsection*{Outline}
The remainder of this paper is organized as follows.
In Section~\ref{sect:2}, we recall some notation and preliminaries on polynomials.
In Section~\ref{sect:3}, we propose a type of bounded degree SOS hierarchy to solve SPLD polynomial optimization problems.
Section~\ref{sect:4} presents numerical experiments illustrating the performance of the proposed hierarchy; moreover, an application to portfolio optimization problems is also studied.
Section~\ref{sect:5} focuses on solving a class of convex SPLD polynomials, and also on an application to convex SPLD polynomial regression.
Conclusions and further remarks are given in Section~\ref{sect:6}.

\section{Preliminaries}\label{sect:2}

We begin this section by fixing some notation and preliminaries.
Suppose $n \in \N$ with $n\geq 1$ (where $\N$ denotes the set of nonnegative integers).
Let $\R^n$ denote the Euclidean $n$-dimensional space, and let $\bx :=(x_1, x_2, \ldots, x_n)$ denote an arbitrary element of $\R^n.$
The nonnegative orthant of $\R^n$ is denoted by
$\R_{+}^n:=\left\{\bx \in \R^n \colon x_i \geq 0,\ i = 1,\ldots, n \right\}.$
Let $\mathcal{S}^n$ be the set of $n\times n$ symmetric matrices with the usual inner product $\langle A, B\rangle := {\rm trace}(AB)$ for $A,B\in \mathcal{S}^n.$
Let $X\in \mathcal{S}^n.$
We say $X$ is a positive semidefinite (PSD) matrix  denoted by $X\succeq 0,$
whenever $\bz^TX\bz\ge 0$ for all $\bz\in\R^n,$ and denote by $\mathcal{S}_+^n$
the set of $n\times n$ symmetric PSD matrices.

The ring of all real polynomials in the variable $\bx$ is denoted by $\R[\bx].$
Moreover, the vector space of all real polynomials in the variable $\bx$ with degree at most $d$ is denoted by $\R[\bx]_d.$
The degree of a polynomial $f$ is denoted by $\deg (f).$
For a multi-index $\ba:=(\alpha_1,\ldots,\alpha_n)\in\N^n,$ let us denote $|\ba|:=\sum_{i=1}^n\alpha_i$ and $\N^n_d:= \{\ba \in \N^n \colon  |\ba|\leq d\}.$
The notation $\bx^\ba$ denotes the monomial $x_1^{\alpha_1}\cdots x_n^{\alpha_n}.$
The canonical basis of $\R[\bx]_d$ is denoted by
\begin{equation}\label{cano_basis}
\lceil \bx\rceil_d:=(\bx^\ba)_{\ba\in\N^n_{d}}=(1,x_1,\ldots,x_n,x_1^2,x_1x_2,\ldots,x_n^2,\ldots,x_1^d,\ldots,x_n^d)^T,
\end{equation}
which has dimension $s(n,d):=\left( \substack{ n+d \\ d }\right).$
We say that a real polynomial $f$ is sum-of-squares (SOS)
if there exist real polynomials $p_j,$ $j = 1,\ldots,q,$ such that $f =\sum_{j=1}^{q}p_{j}^2.$
The set of all SOS real polynomials in the variable $\bx$ is denoted by $\Sigma[\bx].$ Clearly, $\Sigma[\bx]\subset \R[\bx].$
In addition, the set consisting of all SOS real polynomials with degree at most $2d$ is denoted by
$\Sigma[\bx]_{d}$, similarly one has $\Sigma[\bx]_{d}\subset\R[\bx]_{2d}.$
For $f\in\R[\bx]_d,$ we denote by $(f_\ba)$ $(\ba\in\N^n_d)$ the coefficients of $f.$

Given an $s(n, 2d)$-vector $\by:= (y_\ba)_{\ba\in\N^n_{2d}}$ with $y_{\bze}=1,$ let $\M_{d}(\by)$ be the moment matrix of dimension $s(n, d),$ with rows and columns labeled by \eqref{cano_basis} in the sense that
\begin{equation*}
\M_d(\by)(\ba,\bb):=y_{\ba+\bb}, \ \forall \ba,\bb\in\N^n_d.
\end{equation*}
Namely, $\M_d(\by)=\sum_{\ba\in\N^n_{2d}}y_\ba (\lceil \bx\rceil_d\lceil \bx\rceil_d^T)_\ba,$ where $(\lceil \bx\rceil_d\lceil \bx\rceil_d^T)_\ba$ denotes the coefficient matrix corresponding to the monomial $\bx^\ba$ in the polynomial matrix $\lceil \bx\rceil_d\lceil \bx\rceil_d^T.$
For example, for $n=2$ and $d=1,$
\begin{equation*}
\M_d(\by)=\sum_{\ba\in\N^2_{2}}y_\ba\left(
                                      \begin{array}{ccc}
                                        1 & x_1 & x_2 \\
                                        x_1 & x_1^2 & x_1x_2 \\
                                        x_2 & x_1x_2 & x_2^2 \\
                                      \end{array}
                                    \right)_\ba=\left(
                                      \begin{array}{ccc}
                                        y_{00} & y_{10} & y_{01} \\
                                        y_{10} & y_{20} & y_{11} \\
                                        y_{01} & y_{11} & y_{02} \\
                                      \end{array}
                                    \right).
\end{equation*}

The following proposition identifies whether a polynomial can be written as an SOS via semidefinite programming problems.
\begin{proposition}\label{proposition1}{\rm \cite{Lasserre2010,Lasserre2015book}}
A polynomial $f\in\R[\bx]_{2d}$ has an SOS decomposition if and only if there exists $Q\in \mathcal{S}^{s(n,d)}_+$ such that
\begin{equation*}
\langle (\lceil \bx\rceil_d\lceil \bx\rceil_d^T)_\ba ,Q\rangle=f_\ba, \ \forall \ba\in\N_{2d}^n.
\end{equation*}
\end{proposition}

We close this section by recalling the notion of SOS-convex polynomials, introduced by Helton and Nie \cite{Helton2010}, which are a subclass of convex polynomials . Remarkably, the SOS-convexity can be checked numerically by solving an SDP problem (see \cite{Helton2010}), whereas verifying the convexity of a polynomial (with even degree $d\geq 4$) is generally NP-hard (see \cite[Theorem 13.8]{Lasserre2015book}).
\begin{definition}[SOS-convexity \cite{Helton2010,Ahmadi2013}]{\rm
A real polynomial $f$ on $\R^n$ is called {\it SOS-convex} if the Hessian matrix function $H \colon \bx \to\nabla^2f(\bx)$ is an SOS matrix polynomial, that is, there exists a matrix polynomial $F(\bx)$ with $n$ columns such that $\nabla^2f(\bx)=F(\bx)^TF(\bx).$
Equivalently, $f(\bx) - f(\by) - \langle \nabla f(\by), \bx - \by \rangle$ is SOS in variables $(\bx, \by).$
}\end{definition}

\section{SPLD polynomial optimization}\label{sect:3}

In this section, we aim to efficiently solve a class of structured optimization problems involving SPLD polynomials (see Definition~\ref{SPLD_poly}).
We call such a class of optimization problems SPLD polynomial optimization problems, by which we mean a problem consisting of minimizing an SPLD polynomial over a feasible set defined by a finite system of SPLD polynomials. Mathematically speaking, consider the following problem:
\begin{align}\label{P}
\inf\limits_{\bx\in\R^n} \left\{ f_0(\bx) \colon  0\leq f_i(\bx)\leq 1, \ i=1,\ldots,m \right\}, \tag{${\rm P}$}
\end{align}
where $f_i\colon\R^n\to\R,$ $i=0,1,\ldots,m,$ are SPLD polynomials,
that is,
\begin{align*}
f_i(\bx):=s_i(\bx)+l_i(\bx), \ i=0,1,\ldots,m,
\end{align*}
for all $\bx\in\R^n,$ where each $s_i\colon\R^n\to\R$ is a separable polynomial and each $l_i\colon\R^n\to\R$ is a polynomial with $\deg(l_i) < \deg(s_i).$
We denote the feasible set of the problem \eqref{P} by
\begin{align}\label{feasible-k}
{\rm\bf{K}}:=\left\{\bx\in \R^n \colon 0\leq f_i(\bx)\leq 1, \ i=1,\ldots,m\right\}.
\end{align}

In what follows, we make the following assumptions on the feasible set ${\rm\bf{K}}.$
\begin{assumption}\label{assumption1}
We assume that
\begin{itemize}
  \item[{\rm (i)}] ${\rm\bf{K}}$ is a nonempty compact set and the family $\{1,f_1,\ldots,f_m\}$ generates $\R[\bx];$
  \item[{\rm (ii)}] the interior of the feasible set ${\rm\bf{K}}$ is nonempty.
\end{itemize}
\end{assumption}
It is worth noting that any compact feasible set can be rescaled into the set ${\rm\bf{K}}$ (as defined in \eqref{feasible-k}) without loss of generality.
Indeed, for any compact set ${\rm\bf{K}}':=\left\{\bx\in \R^n \colon g_i(\bx)\geq 0, \ i=1,\ldots,m\right\},$ we may define $f_i := g_i/ M,$ where $M > \max\limits_{1\leq i\leq m}\sup\limits_{\bx\in {\rm\bf{K}}'}\{g_i(\bx)\},$ so that ${\rm\bf{K}}'={\rm\bf{K}}.$

To proceed, we briefly recall an important theorem (by Krivine \cite{Krivine1964} and
Stengle \cite{Stengle1974}, respectively)
on the representation of polynomials that are positive on ${\rm\bf{K}}$.
\begin{lemma}[Krivine--Stengle's positivstellensatz \rm{\cite{Lasserre2010,Krivine1964,Stengle1974}}]\label{KS_positivity}
Let $f_i,$ $i =0, 1,\ldots, m,$ be real polynomials and
${\rm\bf{K}} = \{\bx \in\R^n \colon 0\leq f_i(\bx) \leq 1, \ i = 1,\ldots, m\}.$
Suppose that the {\rm Assumption~\ref{assumption1} (i)} holds.
If $f_0(\bx) > 0$ for all $\bx \in {\rm\bf{K}},$ then there exist finitely many nonnegative coefficients $c_{\p,\q}$ such that
\begin{align*}
f_0 - \sum_{\p,\q\in \N^m}c_{\p,\q} \prod_{i=1}^m  f_i^{p_i} (1-f_i)^{q_i}\equiv0.
\end{align*}
\end{lemma}

From now on, for each $i=0,1,\ldots,m,$ we let $s_i(\bx):=\sum_{j=1}^nu_i^j(x_j)$ for all $\bx\in\R^n,$ with univariate polynomials $u_i^j$ in variable $x_j$ and denote by
\begin{itemize}
\item $d_i^j$ the smallest positive integer such that $2d_i^j\geq\deg (u_i^j)$ for $j=1,\ldots,n$;
\item $r_i$ the smallest positive integer such that $2r_i\geq\deg(l_i)$;
\item $d_j,$ $j=1,\ldots,n,$ and $r$ the positive integers such that
\begin{align*}
d_j\geq\max_{i=0,1,\ldots,m}d^j_i, \ j=1,\ldots,n, \ \textrm{and}\  r\geq\max_{i=0,1,\ldots,m}r_i.
\end{align*}
\end{itemize}
Observe that, since we consider SPLD polynomials, one always has $\max\limits_{j=1,\ldots,n}d_j > r.$

\subsection{BSOS-SPLD: Bounded degree SOS hierarchies for SPLD polynomial optimization}
\label{subsection3.1}

In this subsection, we address SPLD polynomial optimization problems using a modified version of the BSOS hierarchy, tailored to the SPLD structure.
For any fixed $k\in\N,$ we first consider the following problem:
\begin{align}\label{Pk}
\inf\limits_{\bx\in\R^n} \left\{ f_0(\bx) \colon  h_{\p,\q}(\bx)\geq 0, \ \forall (\p,\q)\in N_k \right\}, \tag{${\rm P}_k$}
\end{align}
where $N_k:=\{(\p,\q)\in\N^m\times\N^m : |\p|+|\q|\leq k\}$ and  $h_{\p,\q}\colon\R^n\to\R$ is the polynomial defined by
\begin{align*}
h_{\p,\q}(\bx):=\prod_{i=1}^m{f}_i(\bx)^{p_i}(1-{f}_i(\bx))^{q_i}
\end{align*}
for all $\bx\in \R^n.$ The problem \eqref{Pk} is equivalent to the problem \eqref{P} since \eqref{Pk} is obtained by adding redundant constraints, which are identically nonnegative over the feasible set of the problem \eqref{P}; see \cite[Section 2.3]{Lasserre2017} for more details.

Now, we consider the Lagrangian dual of the problem \eqref{Pk}:
\begin{align}\label{Dk}
\sup\limits_{\substack{\mu\in\R,c_{\p,\q}\geq0}}\left\{\mu : f_0-\sum_{(\p,\q)\in N_k}c_{\p,\q}h_{\p,\q}-\mu\geq0, \ \forall \bx\in\R^n\right\}.\tag{${\rm D}_k$}
\end{align}
By weak duality, we have ${\rm val}\eqref{P}={\rm val}\eqref{Pk}\geq{\rm val}\eqref{Dk}.$

To proceed, for fixed $\bfd:=(d_1,\ldots,d_n)\in\N^n$ and $r\in\N,$ we consider the $k$-th hierarchy of SOS relaxation for the problem \eqref{P} as follows:
\begin{align}\label{Dkdr}
\sup\limits_{\substack{\mu\in\R,c_{\p,\q}\geq0, \sigma\in\Sigma[\bx]_{r}, \sigma_j\in\Sigma[x_j]_{d_j} }}\left\{ \mu \colon f_0-\sum_{(\p,\q)\in N_k}c_{\p,\q}h_{\p,\q}-\mu=\sigma+\sum_{j=1}^n\sigma_j\right\}. \tag{${\rm D}_k^{\bfd,r}$}
\end{align}
It follows that ${\rm val}\eqref{P}\geq{\rm val}\eqref{Dk}\geq{\rm val}\eqref{Dkdr}$, since in the problem \eqref{Dkdr}, the term $\sigma+\sum_{j=1}^n\sigma_j$ is nonnegative.

In practical implementations of the equation in the problem \eqref{Dkdr}, there are two primary approaches for handling polynomial equalities within the resulting semidefinite program.
Namely, given two polynomials $p, q \in \mathbb{R}[\bx]_d$, one can assert their equality using one of the following two strategies:
The first is to \emph{equate their coefficients}, typically by requiring $p_\gamma = q_\gamma$ for all $\gamma \in \mathbb{N}^n_d$ in the monomial basis.
The second is to \emph{equate their values}, by evaluating the polynomials at generic sample points, typically drawn at random from a bounded domain such as $[-1, 1]^n$.

Each of the two approaches presents certain challenges.
When coefficients are equated, it necessitates expanding polynomials like $g_j$ and $(1 - g_j)$, which may introduce numerical instability due to the growth of binomial coefficients, particularly at higher relaxation orders $d$.
This can severely affect the conditioning of the resulting systems.
Alternatively, enforcing equality by sampling values may also lead to ill-conditioned linear systems, especially if the evaluation points or the structure of the polynomials result in nearly linearly dependent constraints.

In our implementation of the BSOS-SPLD relaxation, we enforce the polynomial equalities in the problem~\eqref{Dkdr} by equating coefficients directly.
Since the number of variables is kept relatively small while we intend to deal with polynomials of high-degree, this choice is reasonable in the present context.
Still, it should be noted that at higher relaxation orders, the size and complexity of the coefficient comparison can grow substantially, potentially resulting in numerical difficulties similar to those mentioned earlier.

Let $\dd$ be the smallest positive integer such that
\begin{equation*}
2\dd\geq\max\{\deg(f_0),\ k\max_{1\leq i\leq m}\{\deg (f_i)\}\}.
\end{equation*}
Let $\sigma\in\Sigma[\bx]_{r},$ and let $\sigma_j\in\Sigma[x_j]_{d_j},$ $j=1,\ldots,n.$
Then it follows from Proposition~\ref{proposition1} that there exist matrices $X\in \mathcal{S}^{s(n,r)}_+$  and $X_j\in \mathcal{S}^{d_j+1}_+,$ $j=1,\ldots,n,$ such that
\begin{align*}
\sigma(\bx)&=\langle\lceil \bx\rceil_r\lceil \bx\rceil_r^T,X\rangle=\sum_{\ba\in \N^n_{2\dd}}\bx^\ba\langle B_\ba,X\rangle,\\
\sigma_j(x_j)&=\langle\lceil x_j\rceil_{d_j}\lceil x_j\rceil_{d_j}^T,X_j\rangle=\sum_{\substack{\ba\in \N^n_{2\dd}}}\bx^{\ba}\langle B_{\ba}^j,X_j\rangle,
\end{align*}
where $B_\ba$ and each $B_{\ba}^j$ are symmetric matrices defined by
\begin{equation}\label{Ba1}
B_\ba:=\left(\lceil \bx\rceil_r\lceil \bx\rceil_r^T\right)_\ba \textrm{ and } B_{\ba}^j:=\left(\lceil x_j\rceil_{d_j}\lceil x_j\rceil_{d_j}^T\right)_{\ba}
\end{equation}
for all $\ba\in \N^n_{2\dd}.$
With this fact, we now reformulate the $k$-th hierarchy \eqref{Dkdr} as the following semidefinite programming:
\begin{align*}
\sup\limits_{\substack{\mu,c_{\p,\q}, X, X_j}} \bigg\{ \mu\in\R : \, & \bigg(f_0-\sum_{(\p,\q)\in N_k}c_{\p,\q}h_{\p,\q}-\mu\bigg)_\ba=\langle B_\ba, X\rangle+\sum_{j=1}^n\langle B_{\ba}^j, X_j\rangle, \ \forall \ba\in \N^n_{2\dd},\\
&c_{\p,\q}\geq0, \ (\p,\q)\in N_k, \  X\in \mathcal{S}^{{s(n,r)}}_+, \  X_j\in \mathcal{S}^{d_j+1}_+, \  j=1,\ldots,n\bigg\}.\notag
\end{align*}

The Lagrangian dual of the above semidefinite program, also an SDP, is given by the following problem:
\begin{align}\label{LDkdr}
\inf\limits_{\substack{\by\in\R^{s(n,2\dd)}}}\quad & \sum_{\ba\in\N^n_{2\dd}}(f_0)_\ba y_\ba \tag{$\widetilde{\rm D}_k^{\bfd,r}$} \\
 {\rm s.t. } \qquad\ \ &  \sum_{\ba\in\N^n_{2\dd}}(h_{\p,\q})_\ba y_\ba\geq0, \  (\p,\q)\in N_k, \nonumber \\
 &\sum_{\ba\in\N^n_{2\dd}}y_\ba (\lceil x_j\rceil_{d_j}\lceil x_j\rceil_{d_j}^T)_\ba\succeq0, \ j=1,\ldots,n, \nonumber \\
 &\sum_{\ba\in\N^n_{2\dd}}y_\ba (\lceil \bx\rceil_r\lceil \bx\rceil_r^T)_\ba\succeq0,  \nonumber \\
 &y_{\bze}=1.\nonumber
\end{align}

The following results show that the sequence of the optimal values associated with the problems \eqref{Dkdr} and \eqref{LDkdr} asymptotically converges to the global minimum of the problem \eqref{P}.
This result can in fact be derived from \cite[Theorem~2]{Lasserre2017}.
Nevertheless, in order to reflect the structured and computational distinctions in the SOS representation adopted in the BSOS-SPLD hierarchy, we will provide a full proof for completeness.
\begin{theorem}\label{thm1}
Let $\bfd\in\N^n$ and $r\in\N$ be fixed.
Suppose that {\rm Assumption~\ref{assumption1} (i)} holds. Then$,$ ${\rm val}\eqref{Dkdr}\leq{\rm val}\eqref{LDkdr}\leq{\rm val}\eqref{P}$ for all $k\in\N,$ and
\begin{align*}
\lim_{k\to\infty}{\rm val}\eqref{Dkdr}=\lim_{k\to\infty}{\rm val}\eqref{LDkdr}={\rm val}\eqref{P}.
\end{align*}
\end{theorem}
\begin{proof}
The proof is given in the Appendix.
\end{proof}

At this point, we would like to mention that strong duality holds between the problems \eqref{Dkdr} and \eqref{LDkdr} under Assumption~\ref{assumption1}, i.e.,
\begin{align*}
{\rm val}\eqref{Dkdr}={\rm val}\eqref{LDkdr}
\end{align*}
for every $k\in\ \N,$ and the problem \eqref{Dkdr} has an optimal solution; see e.g., \cite[Lemma~1]{Lasserre2017}.

\begin{remark}[Comparison with the standard BSOS hierarchy \cite{Lasserre2017}]{\rm
In the standard BSOS hierarchy~\cite{Lasserre2017}, the semidefinite constraint in each relaxation involves a single SOS polynomial $\sigma\in\Sigma[\bx]_d$, so the associated PSD matrix has a fixed size $s(n,d)$, where $2d$ is typically chosen to be at least the largest degree among all involved polynomials.
In contrast, the BSOS-SPLD hierarchy~\eqref{Dkdr} replaces $\sigma$ by the structured decomposition
\begin{equation*}
\sigma+\sum_{j=1}^n \sigma_j, \quad \sigma\in\Sigma[\bx]_r,\ \sigma_j\in\Sigma[x_j]_{d_j},
\end{equation*}
so that the lower-degree part is handled by a single multivariate block while the high-degree separable terms are treated via smaller univariate blocks.
This may lead to substantially smaller PSD matrices in the resulting SDP formulations; see Section~\ref{sect:4} for the concrete choices of $(r,d_1,\ldots,d_n)$ used in the numerical experiments.
}\end{remark}

The following result provides a sufficient condition for finite convergence of the BSOS-SPLD hierarchy and guarantees that at least one minimizer can be extracted.
The proof follows a similar line of reasoning as in \cite[Lemma~1]{Lasserre2017}.
However, due to the structured modifications presented in the BSOS-SPLD hierarchy, and with a slight abuse of notation, we will give a full proof for the sake of completeness.
\begin{theorem}[{cf. \cite[Lemma~1]{Lasserre2017}}]\label{thm2}
Let $\bfd\in\N^n$ and $r\in\N$ be fixed$,$ and let ${\bar s}\in \N$ and ${\bar l}\in\N$ be smallest integers such that $2\bar s\geq \max_{0\leq i\leq m}\{\deg (s_i)\}$ and $2\bar l\geq \max_{0\leq i\leq m}\{\deg (l_i)\},$ respectively.
Assume that $\overline \by \in\R^{s(n,2\dd)}$ is an optimal solution to the problem \eqref{LDkdr}.
If{\small
\begin{equation}\label{rel1_thm2}
\max\left\{\left\{{\rm rank}\left(\sum_{\ba\in\N^n_{2\dd}}\overline y_\ba ({\lceil x_j\rceil_{\bar s}}{\lceil x_j\rceil_{\bar s}}^T)_\ba\right)\right\}_{1\leq j\leq n}, \ {\rm rank}\left(\sum_{\ba\in\N^n_{2\dd}}\overline y_\ba ({\lceil \bx\rceil_{\bar l}}{\lceil \bx\rceil_{\bar l}}^T)_\ba\right)\right\}=1,
\end{equation}}
then $\overline \bx:=(\overline y_\ba)_{|\ba|=1}\in\R^n$ is an optimal solution to the problem \eqref{P}.
\end{theorem}
\begin{proof}
See the Appendix for the proof.
\end{proof}

Note that the condition \eqref{rel1_thm2} can be replaced by the following standard rank-1 condition
\begin{equation*}
{\rm rank}\left(\sum_{\ba\in\N^n_{2\dd}}\overline y_\ba ({\lceil \bx\rceil_{\bar s}}{\lceil \bx\rceil_{\bar s}}^T)_\ba\right)=1,
\end{equation*}
as stated in \cite[Lemma~1]{Lasserre2017}.
The condition \eqref{rel1_thm2} may be used to reduce the computational cost of certifying optimality, compared with checking the standard rank-1 condition.
Moreover, if at some relaxation order the optimal moment matrix satisfies the standard flat extension (flat truncation) condition, together with the usual positive semidefiniteness conditions for the moment and localizing matrices, then the BSOS-SPLD relaxation is exact at that order and the hierarchy exhibits finite convergence.
In this case, the flat extension condition guarantees the existence of a finitely atomic representing measure for the optimal truncated moment sequence, with support contained in the feasible set (as certified by the localizing conditions).
Since the relaxation is exact, these support points are global minimizers and can be extracted by standard linear-algebra procedures, with their number equal to the rank of the moment matrix at the order where the flat extension condition holds (see, e.g., \cite[Lemma~1]{Lasserre2017}).
Indeed, a flat extension–type condition tailored to SPLD polynomial problems can be used to certify optimality at lower computational cost.

\section{Numerical experiments}\label{sect:4}

In this section, we present several numerical examples to illustrate the proposed bounded degree hierarchies.
All numerical experiments were conducted on a desktop PC running Windows 11 (64-bit) with an Intel Core i5-10400F (2.90GHz) CPU and 16\,GB RAM, using MATLAB R2023a with YALMIP \cite{Lofberg2004} (release R20250626) and the SDP solver MOSEK \cite{MOSEK2025} (version 11.0.29).
All runs were performed using standard YALMIP settings and default MOSEK parameter settings.

The tables (Table~\ref{table1}, Table~\ref{table2}, and Table~\ref{table3}) in this section provide the following information:
\begin{itemize}
\item $(d, k)$ stands for fixed degree $d$ and relaxation order $k$ in the BSOS hierarchy; and $(d_0,r,k)$ means fixed degrees $d_0$ for the separable term, $r$ for the lower-degree term, and relaxation order $k$ in the BSOS-SPLD hierarchy.
\item The abbreviation $\verb"Opt"$ denotes the optimal value obtained by the BSOS hierarchy and BSOS-SPLD hierarchy at order $k.$
\item $\verb"Time"$ denotes the average computation time (in seconds) over 10 runs, including the time for variable declaration in YALMIP, SDP formulation and solving, and verification of the optimality condition.
\item The abbreviations $\verb"ms"$ and $\verb"max_ms"$ refer to the size of the moment matrix $\M_d(\by)$ in BSOS hierarchy and the maximum size over the moment matrix $\M_r(\by)$ and all moment matrices $\M_{d_0}(\by^j)$ in BSOS-SPLD hierarchy, respectively.
\item The abbreviation $\verb"rk"$ refers to the rank of the moment matrix in BSOS hierarchy. The abbreviation $\verb"max_rk"$ denotes the maximum rank among the moment matrices involved in \eqref{rel1_thm2} of the BSOS-SPLD hierarchy\footnote{We declare that the matrix has numerical rank-one if the largest eigenvalue for each matrix is at least $10^4$ times larger than the second largest eigenvalue.}.
\item A dash ``$-$'' indicates that the problem could not be solved within the available time/memory limits (out of memory).
\end{itemize}
At this point, we would like to note that the reported computation times can depend on the SDP solver and the implementation used to formulate the SDP.
Accordingly, alternative solvers (e.g., SDPNAL+ \cite{Sun2020} or SDPT3 \cite{Toh1999}) may yield different computation times, while the structured benefit of BSOS-SPLD is expected to remain.

In all experiments we adopt the following simple and practical rule for choosing the degrees $r$ and $d_1,\ldots,d_n$ in the BSOS-SPLD hierarchy.
\begin{itemize}
    \item We fix $r$ to be at least the degree of the lower-degree polynomial term.
    \item Then, for each $j$, we choose $d_j$ as follows:
    \begin{itemize}
        \item[{\rm (i)}] If $s(n, r)$ is greater than or equal to the degree of the separable polynomial in variable $x_j,$ then set $d_j + 1 = s(n, r).$
        In a few instances we allow some $d_j$ to be slightly larger than $s(n,r)$, which increases the PSD size but may yield the optimal value at a lower relaxation order $k$ (see, e.g., the choices $d_0=27$ and $d_0=39$ for problem ${\rm P}_{6,8}$ in Table~\ref{table1}).
        \item[{\rm (ii)}] Otherwise$,$ choose $d_j$ to be greater than or equal to the degree of the separable polynomial in $x_j.$
    \end{itemize}
\end{itemize}
This choice simplifies the degree selection process while keeping the size of the semidefinite constraints fixed throughout the hierarchy, thereby limiting the impact of degree selection on the overall computational cost.

\subsection{BSOS versus BSOS-SPLD}

We consider the following family of nonconvex polynomial optimization problems \cite{Lasserre2017}:
\begin{align}\label{Pnd}
\min_{\bx\in \mathbb{R}^n} \quad & f_{n,q}(\bx) \tag{${\rm P}_{n,q}$} \\
{\rm s.t. } \quad \ & 0\leq f_{i}(\bx)\leq 1,\ i=1,\ldots,5,\notag\\
& 0\leq x_j\leq 1, \  j=1,\ldots,n.\notag
\end{align}
where $n$ and $q$ are even integers, and the objective is
\begin{equation*}
f_{n,q}(\bx)=\sum_{j=1}^{n}(-1)^{j+1}x_j^{q}+x_1-x_2,
\end{equation*}
and the constraint polynomials $f_i,$ $i=1,\ldots,5,$ are given by
\begin{align*}
f_{1}(\bx) &= \sum_{p=1}^{n/2} \big(2x_{2p-1}^{q}+3x_{2p}^{2}+2x_{2p-1} x_{2p}\big), \ \ f_{2}(\bx) = \sum_{p=1}^{n/2} \big(3x_{2p-1}^{2}+2x_{2p}^{2}-4x_{2p-1} x_{2p}\big),\\
f_{3}(\bx) &= \sum_{p=1}^{n/2} \big(x_{2p-1}^{2}+6x_{2p}^{2}-4x_{2p-1} x_{2p}\big), \ \  f_{4}(\bx) = \sum_{p=1}^{n/2} \big(x_{2p-1}^{2}+4x_{2p}^{q}-3x_{2p-1} x_{2p}\big),\\
f_{5}(\bx) &= \sum_{p=1}^{n/2} \big(2x_{2p-1}^{2}+5x_{2p}^{2}+3x_{2p-1} x_{2p}\big).
\end{align*}

\begin{table}[t]
\caption{Comparison of BSOS and BSOS-SPLD for several problems~\eqref{Pnd}}
\label{table1}
\centering{\footnotesize
\begin{tabular}{ccccccccccc}
\toprule
\multicolumn{1}{l}{\textbf{Problem}} & \multicolumn{5}{l}{\textbf{BSOS}}  & \multicolumn{4}{l}{\textbf{BSOS-SPLD}} \\
\cmidrule(rl){2-6} \cmidrule(rl){7-11}
\eqref{Pnd}         & $(d,k)$ & $\verb"Opt"$ & $\verb"Time"$ & $\verb"ms"$ & \verb"rnk" &$(d_0,r,k)$ & $\verb"Opt"$ & $\verb"Time"$ & $\verb"max_ms"$ & \verb"max_rnk" \\
\toprule
\toprule
$({\rm P}_{6,6})$   & $3,1$   & $-0.6597$    & $0.4$         & $84$        & $33$       & $27,2,1$   & $-0.6153$    & $0.2$         & $28$            & $4$            \\
                    & $3,2$   & $-0.6597$    & $0.6$         & $84$        & $25$       & $27,2,2$   & $-0.4129$    & $0.4$         & $28$            & $1$            \\
                    & $3,3$   & $-0.4129$    & $4.7$         & $84$        & $1$        &            &              &               &                 &                \\
\midrule
$({\rm P}_{6,8})$   & $4,1$   & $-0.6597$    & $5.2$         & $210$       & $55$       & $27,2,1$   & $-0.6464$    & $0.2$         & $28$            & $5$            \\
                    & $4,2$   & $-0.6597$    & $5.6$         & $210$       & $48$       & $27,2,2$   & $-0.4109$    & $0.5$         & $28$            & $5$            \\
                    & $4,3$   & $-0.6597$    & $9.9$         & $210$       & $48$       & $27,2,3$   & $-0.4090$    & $4.3$         & $28$            & $1$            \\
                    & $4,4$   & $-0.4090$    & $102.1$       & $210$       & $1$        & $39,2,1$   & $-0.6251$    & $0.3$         & $40$            & $4$            \\
                    &         &              &               &             &            & $39,2,2$   & $-0.4090$    & $0.4$         & $40$            & $1$            \\
\midrule
$({\rm P}_{6,10})$  & $5,1$   & $-0.6597$    & $33.4$        & $462$       & $89$       & $39,2,1$   & $-0.6403$    & $0.3$         & $40$            & $4$            \\
                    & $5,2$   & $-0.6597$    & $39.7$        & $462$       & $84$       & $39,2,2$   & $-0.4099$    & $0.6$         & $40$            & $5$            \\
                    & $5,3$   & $-0.6597$    & $49.2$        & $462$       & $78$       & $39,2,3$   & $-0.4084$    & $4.9$         & $40$            & $1$            \\
                    & $5,4$   & $-0.6597$    & $153.6$       & $462$       & $78$       & $61,2,1$   & $-0.5944$    & $0.4$         & $62$            & $7$            \\
                    & $5,5$   & $-$          & $-$           & $462$       & $-$        & $61,2,2$   & $-0.4084$    & $0.6$         & $62$            & $1$            \\
\midrule
\midrule
$({\rm P}_{8,6})$   & $3,1$   & $-0.6597$    & $3.9$         & $165$       & $54$       & $44,2,1$   & $-0.5547$    & $0.5$         & $45$            & $5$            \\
                    & $3,2$   & $-0.6597$    & $4.6$         & $165$       & $64$       & $44,2,2$   & $-0.4129$    & $0.6$         & $45$            & $1$            \\
                    & $3,3$   & $-0.4129$    & $15.4$        & $165$       & $1$        &            &              &               &                 &                \\
\midrule
$({\rm P}_{8,8})$   & $4,1$   & $-0.6597$    & $106.4$       & $495$       & $114$      & $44,2,1$   & $-0.6028$    & $0.5$         & $45$            & $5$            \\
                    & $4,2$   & $-0.6597$    & $97.1$        & $495$       & $125$      & $44,2,2$   & $-0.4090$    & $0.9$         & $45$            & $1$            \\
                    & $4,3$   & $-0.6597$    & $133.6$       & $495$       & $125$      &            &              &               &                 &                \\
                    & $4,4$   & $-0.4090$    & $558.1$       & $495$       & $1$        &            &              &               &                 &                \\
\midrule
$({\rm P}_{8,10})$  & $5,1$   & $-$          & $-$           & $1287$      & $-$        & $61,2,1$   & $-0.5979$    & $0.6$         & $62$            & $5$            \\
                    & $5,2$   & $-$          & $-$           & $1287$      & $-$        & $61,2,2$   & $-0.4093$    & $1.9$         & $62$            & $1$            \\
\midrule
\midrule
$({\rm P}_{10,6})$  & $3,1$   & $-0.6597$    & $28.0$        & $286$       & $97$       & $65,2,1$   & $-0.5112$    & $1.4$         & $66$            & $6$            \\
                    & $3,2$   & $-0.6597$    & $33.1$        & $286$       & $115$      & $65,2,2$   & $-0.4129$    & $1.4$         & $66$            & $1$            \\
                    & $3,3$   & $-0.4129$    & $63.7$        & $286$       & $1$        &            &              &               &                 &                \\
\midrule
$({\rm P}_{10,8})$  & $4,1$   & $-$          & $-$           & $1001$      & $-$        & $65,2,1$   & $-0.5466$    & $1.2$         & $66$            & $6$            \\
                    & $4,2$   & $-$          & $-$           & $1001$      & $-$        & $65,2,2$   & $-0.4090$    & $3.2$         & $66$            & $1$            \\
\midrule
$({\rm P}_{10,10})$ & $5,1$   & $-$          & $-$           & $3003$      & $-$        & $65,2,1$   & $-0.5895$    & $1.7$         & $66$            & $6$            \\
                    & $5,2$   & $-$          & $-$           & $3003$      & $-$        & $65,2,2$   & $-0.4084$    & $12.5$        & $66$            & $1$            \\
\bottomrule
\end{tabular}}
\end{table}

Now, we solve the $k$-th order relaxations of problem~\eqref{Pnd} using both the standard {\bf BSOS} hierarchy \cite{Lasserre2017} and the proposed {\bf BSOS-SPLD} hierarchy.

Table~\ref{table1} compares optimal values and CPU times of the relaxations.
The main observations are as follows:
\begin{enumerate}[\upshape (a)]
\item For $({\rm P}_{6,6})$, with $(d,k)=(3,3)$ the BSOS hierarchy attains $\verb"rnk"=1$ in $4.7$ seconds.
In contrast, the BSOS-SPLD hierarchy with $(d_0,r,k)=(27,2,2)$ already achieves $\verb"max_rnk"=1$ in $0.4$ second.
This indicates that, for a suitable fixed degree $d_0$, our method can reach the same optimal value as BSOS with a lower computational cost.

\item For $({\rm P}_{6,8})$, the BSOS hierarchy requires substantially more time as the relaxation order increases:
although the optimal value is obtained only at the $4$th relaxation, it takes $102.1$ seconds.
By contrast, under the BSOS-SPLD hierarchy with $d_0=27$ (which satisfies $s(3,2)=d_0+1$), the optimal value is obtained at the $3$rd relaxation in $4.3$ seconds, and increasing $d_0$ to $39$ allows the method to reach the same optimal value already at the $2$nd relaxation order in only $0.4$ seconds.
This example illustrates that a larger fixed degree $d_0$ may accelerate convergence, even at the cost of a larger semidefinite matrix size.

\item For the higher-degree case $({\rm P}_{6,10})$, the BSOS hierarchy fails to improve beyond the value $-0.6597$ for all tested relaxation orders $k=1,\ldots,4$, despite CPU times growing up to $153.6$ seconds.
In contrast, the BSOS-SPLD hierarchy with $(d_0,r,k)=(39,2,3)$ or $(61,2,2)$ reaches the optimal value in less than $5$ seconds.
Thus, even for $q=10$, the BSOS-SPLD hierarchy can attain the optimal value at low relaxation orders, while BSOS fails to do so for any of the tested relaxation orders.

\item For the examples with $n=8, 10$, the standard BSOS hierarchy quickly becomes numerically intractable when $q$ increases.
For $({\rm P}_{8,6})$ the $3$rd relaxation (with $d=3$) is still solvable, but already requires $15.4$ seconds.
For $({\rm P}_{8,8})$ the $4$th relaxation (with $d=4$) takes $558.1$ seconds.
For $({\rm P}_{8,10})$, $({\rm P}_{10,8}),$ and $({\rm P}_{10,10})$, the BSOS relaxations cannot be completed because the solver runs out of memory.
In all of these cases, the BSOS-SPLD hierarchy still reaches the optimal value at $2$nd relaxation order.

\item
Overall, Table~\ref{table1} suggests that the computational behaviour of the BSOS-SPLD hierarchy is relatively insensitive to the highest degree $q$ of the separable terms, but more sensitive to the dimension $n$.
For fixed $n$ (e.g., $n=6$), increasing $q$ from $6$ to $10$ only causes a moderate increase in the BSOS-SPLD runtimes, whereas for fixed $q$ the CPU time grows more noticeably when $n$ increases from $6$ to $8$ and $10$.
In contrast, the standard BSOS hierarchy is highly sensitive both to $n$ and to $q$: as either the number of variables or the degree increases, the size of the semidefinite matrices grows rapidly that the SDP relaxations become numerically intractable.

In addition, higher relaxation orders significantly increase not only the cost of solving the SDP but also the time required to generate the problem in YALMIP (variable declaration and SDP construction).
These observations suggest that, for this class of SPLD test problems, fixing suitable values of $d_0$ and $r$ in the BSOS-SPLD hierarchy allows one to reach the optimal value at lower relaxation orders and with significantly reduced overall computation time compared to the standard BSOS hierarchy.
\end{enumerate}

\medskip
\noindent
{\bf Test problem} ({\verb"SPM"})
\begin{align*}
\min\limits_{\bx\in\R^2} \quad & x_1^N + x_2^N +m(x_1,x_2)   \\
{\rm s.t.} \quad \,& 0\leq x_1^2+x_2^2\leq1,\\
&0\leq x_j\leq1, \ j=1,2,
\end{align*}
where $m\colon\R^2\to\R$ is the Motzkin polynomial such that $m(x_1,x_2):=x_1^4x_2^2+x_1^2x_2^4-3x_1^2x_2^2,$ and $N$ is an even integer with $N\geq8.$
Note that the objective $x_1^N + x_2^N + m(x_1,x_2)$ is an SPLD polynomial but not an SPQ polynomial.

We now solve the $k$-th $(k \in\{ 1, 2,5, 6,7\})$ relaxation problems of the test problem ({\verb"SPM"}) for $N\in\{20,40,60,\allowbreak 100,\allowbreak 200,400\}$ using the {\bf BSOS} hierarchy \cite{Lasserre2017} and our {\bf BSOS-SPLD} hierarchy, respectively.

\begin{table}[t]
\caption{Comparison of BSOS and BSOS-SPLD for SPM}
\label{table2}
\centering{\footnotesize
\begin{tabular}{ccccccccccc}
\toprule
\multicolumn{1}{l}{\textbf{Problem}} & \multicolumn{5}{l}{\textbf{BSOS}}  & \multicolumn{4}{l}{\textbf{BSOS-SPLD}} \\
\cmidrule(rl){2-6} \cmidrule(rl){7-11}
$N$   & $(d,k)$ & $\verb"Opt"$ & $\verb"Time"$ & $\verb"ms"$ & \verb"rnk" & $(d_0,r,k)$  & $\verb"Opt"$ & $\verb"Time"$ & $\verb"max_ms"$ & \verb"max_rnk" \\
\toprule
$20$  & $10,1$  & $-0.5325$    & $0.1$         & $66$        & $5$        & $10,3,1$     & $-0.5325$    & $0.1$         & $11$            & $5$ \\
      & $10,2$  & $-0.4980$    & $0.1$         & $66$        & $4$        & $10,3,2$     & $-0.4980$    & $0.1$         & $11$            & $4$ \\
      & $10,5$  & $-0.4980$    & $0.6$         & $66$        & $4$        & $10,3,6$     & $-0.4980$    & $1.6$         & $11$            & $2$ \\
      & $10,6$  & $-0.4980$    & $1.7$         & $66$        & $1$        & $10,3,7$     & $-0.4980$    & $5.5$         & $11$            & $1$ \\
\midrule
$40$  & $20,1$  & $-0.5597$    & $1.6$         & $242$       & $5$        & $20,3,1$     & $-0.5597$    & $0.1$         & $21$            & $5$ \\
      & $20,2$  & $-0.5000$    & $1.5$         & $242$       & $4$        & $20,3,2$     & $-0.5000$    & $0.1$         & $21$            & $4$ \\
      & $20,5$  & $-0.5000$    & $2.5$         & $242$       & $4$        & $20,3,6$     & $-0.5000$    & $1.6$         & $21$            & $2$ \\
      & $20,6$  & $-0.5000$    & $3.7$         & $242$       & $1$        & $20,3,7$     & $-0.5000$    & $5.6$         & $21$            & $1$ \\
\midrule
$60$  & $30,1$  & $-0.5623$    & $8.9$         & $496$       & $5$        & $30,3,1$     & $-0.5623$    & $0.1$         & $31$            & $5$ \\
      & $30,2$  & $-0.5000$    & $8.1$         & $496$       & $4$        & $30,3,2$     & $-0.5000$    & $0.1$         & $31$            & $4$ \\
      & $30,5$  & $-0.5000$    & $11.9$        & $496$       & $4$        & $30,3,6$     & $-0.5000$    & $1.6$         & $31$            & $2$ \\
      & $30,6$  & $-0.5000$    & $13.9$        & $496$       & $1$        & $30,3,7$     & $-0.5000$    & $5.6$         & $31$            & $1$ \\
\midrule
$100$ & $50,1$  & $-0.5625$    & $347.9$       & $1326$      & $5$        & $50,3,1$     & $-0.5625$    & $0.1$         & $51$            & $5$ \\
      & $50,2$  & $-0.5000$    & $296.1$       & $1326$      & $4$        & $50,3,2$     & $-0.5000$    & $0.1$         & $51$            & $4$ \\
      & $50,5$  & $-0.5000$    & $472.2$       & $1326$      & $4$        & $50,3,6$     & $-0.5000$    & $1.6$         & $51$            & $2$ \\
      & $50,6$  & $-0.5000$    & $544.1$       & $1326$      & $1$        & $50,3,7$     & $-0.5000$    & $5.7$         & $51$            & $1$ \\
\midrule
$200$ & $100,1$ & $-$          & $-$           & $5151$      & $-$        & $100,3,7$    & $-0.5000$    & $6.1$         & $101$           & $1$ \\
\midrule
$400$ & $200,1$ & $-$          & $-$           & $20301$     & $-$        & $200,3,7$    & $-0.5000$    & $8.6$         & $201$           & $1$ \\
\bottomrule
\end{tabular}}
\end{table}

From the numerical results in Table~\ref{table2}, we observe that for $N=20,$ the BSOS hierarchy reaches the rank-one condition slightly faster ($1.7$ seconds at the $6$th relaxation order) than BSOS-SPLD ($5.7$ seconds at the $7$th relaxation order), though both remain within a low computational cost range.
However, as $N$ increases, the CPU times taken by the BSOS hierarchy also increase rapidly, while the CPU times taken by the BSOS-SPLD hierarchy increase very slowly.
This shows that the BSOS hierarchy is heavily affected by the degree of the involved polynomials; moreover, as the degree increases, the method may no longer be effective.
On the other hand, our BSOS-SPLD hierarchy is more stable, being less affected by the higher-degree terms of the involved polynomials.

In addition, we observe that the BSOS hierarchy reaches the optimal value at the $6$th hierarchy level by satisfying the rank-one condition, whereas the BSOS-SPLD hierarchy requires the $7$th level to satisfy the same condition and obtain the optimal value.
This indicates that although the BSOS-SPLD hierarchy may require a slightly higher hierarchy level to guarantee global optimality, it becomes significantly more efficient than BSOS as $N$ grows, particularly in terms of computational time.

\subsection{An application to portfolio optimization problems}\label{subsec:4}

We consider a synthetic portfolio optimization problem that includes covariance shrinkage and higher–order penalization terms, and serves as a benchmark for nonlinear polynomial optimization.
Given the number of assets $n$ and the sample length $T$, we generate random returns, estimate a covariance matrix, apply correlation shrinkage, and add higher–order penalty terms to obtain a nonconvex polynomial optimization problem.

For a given pair $(n,T),$ we generate synthetic excess returns $R\in \mathbb{R}^{T\times n}$ by drawing mutually independent Gaussian samples
\begin{equation*}
R_{tj} := 0.001\,\xi_{tj},\quad   \xi_{tj} \sim \mathcal N(0,1),\quad t=1,\ldots,T,\ j=1,\ldots,n.
\end{equation*}
The sample mean $\boldsymbol\mu\in\mathbb{R}^n$ and centered returns $X\in\mathbb{R}^{T\times n}$ are defined by
\begin{equation*}
\mu_j := \frac{1}{T} \sum_{t=1}^TR_{tj}, \quad X_{tj} := R_{tj} - \mu_{j}, \quad t=1,\ldots,T,\ j=1,\ldots,n,
\end{equation*}
and the sample covariance $S=(S_{ij})$ is given by
\begin{equation*}
S := \frac{1}{T-1} X^T X \in \mathbb{R}^{n\times n}.
\end{equation*}
In all numerical experiments, we fix the sample length at $T=300$ and vary the portfolio dimension $n$.

Next we apply a constant-correlation shrinkage scheme of Ledoit--Wolf type~\cite{Ledoit2004}.
Let $C$ denote the sample correlation matrix of $R$ and define:
\begin{equation*}
\bar r := \frac{\sum_{i\neq j} C_{ij}}{n(n-1)},\quad   s_i := \sqrt{S_{ii}},\ i=1,\ldots,n,
\end{equation*}
and let $s = (s_1,\ldots,s_n)^T \in \R^n$.
The target matrix  is the highly structured constant-correlation covariance matrix.
It is defined to preserve the sample variances ($s_i^2$) while assuming a single, uniform correlation ($\bar{r}$) for all asset pairs.
Let $T_0 = (t_{ij})$ be the target covariance matrix, where
\begin{equation*}
t_{ij} :=
\begin{cases}
s_i^2, & \text{if } i = j, \\
\bar r \cdot s_i s_j, & \text{if } i \neq j.
\end{cases}
\end{equation*}
Following Ledoit--Wolf~\cite{Ledoit2004}, we compute the optimal shrinkage intensity $\alpha^*$.
This requires estimating two quantities:
(i) the distance between the sample and the target
\begin{equation*}
\hat{d}^2 = \sum_{i=1}^n \sum_{j=1}^n (S_{ij} - t_{ij})^2,
\end{equation*}
(ii) the sum of asymptotic variances of the sample covariances
\begin{equation*}
\hat{b}^2 = \frac{1}{T^2} \sum_{t=1}^T \sum_{i=1}^n \sum_{j=1}^n\left( X_{t i} X_{t j} - S_{ij} \right)^2.
\end{equation*}
The optimal shrinkage intensity $\alpha^*$ that minimizes the mean squared error is then estimated by:
\begin{equation*}
\alpha^* = \max\left(0, \min\left(1, \hat{b}^2/ \hat{d}^2\right)\right).
\end{equation*}
The shrunk covariance matrix is
\begin{equation*}
\Sigma_0 := (1-\alpha^*) S + \alpha^* T_0.
\end{equation*}
To obtain a genuinely nonconvex test problem, we introduce an indefinite quadratic part by shifting and normalizing $\Sigma_0$.
Let $\lambda_{\min}$ and $\lambda_{\max}$ denote the smallest and largest eigenvalues of $\Sigma_0.$
Fix a parameter $\theta\in(0,1)$ and set
\begin{equation*}
  \tau := \lambda_{\min} + \theta\bigl(\lambda_{\max}-\lambda_{\min}\bigr),\qquad
  Q_0 := \Sigma_0 - \tau I_n,\qquad
  Q := \frac{Q_0}{\max_i |\lambda_i(Q_0)|}.
\end{equation*}
The scalar $\tau$ shifts the spectrum of $\Sigma_0$ so that $Q$ becomes indefinite, while the normalization ensures that the eigenvalues of $Q$ remain in $[-1,1]$.

We formulate the optimization problem by defining a trade-off between the mean return $\boldsymbol\mu^T \bx$ and the indefinite quadratic risk term $\bx^T Q \bx,$ together with higher-order penalty terms.
We consider a portfolio with weight vector $\bx\in\R^n$ constrained to lie in the simplex
\begin{equation*}
  \Delta := \{\,\bx\in\R^n : e^T \bx = 1,\ \bx\ge 0\,\}.
\end{equation*}
For fixed parameters $\lambda\ge 0$, $\eta\ge 0$ and an integer $p\ge 2$ (in the experiments we use $\eta=0.25$, $p=4$), we define the SPLD objective
\begin{equation}\label{eq:spld-portfolio}
  f(\bx) = \lambda\sum_{i=1}^n x_i^{2p} + \bx^T Q \bx - \boldsymbol\mu^T \bx + \eta\sum_{i=1}^n\bigl(x_i^2 - 1/n\bigr)^2,
\end{equation}
and the associated portfolio problem
\begin{equation}\label{eq:spld-portfolio-problem}
  \min_{\bx\in\Delta}\; f(\bx).
\end{equation}
The first and last terms in \eqref{eq:spld-portfolio} are even–degree, separable penalties.
The term $\lambda \sum_{i} x_i^{2p}$ discourages overly concentrated positions,
while the penalty $\eta \sum_{i=1}^n (x_i^2 - 1/n)^2$ further regularizes the weights by pulling the squared positions toward a common scale, thereby favouring diversified portfolios within the simplex.
The middle term $\bx^T Q \bx - \boldsymbol\mu^T \bx$ plays the role of a (shifted) mean–variance trade–off with an indefinite risk matrix $Q$.
Overall, $f$ has degree $2p$ (equal to $8$ in our experiments with $p=4$) and admits an SPLD decomposition $f(\bx)=s(\bx)+p(\bx)$ with $s$ separable and $p$ of degree at most $2$.

From a portfolio perspective, solving the portfolio problem~\eqref{eq:spld-portfolio-problem} means choosing a fully invested portfolio $\bx\in\Delta$ that balances expected excess return, indefinite quadratic risk, and high–degree diversification penalties.
In our numerical experiments, this synthetic SPLD portfolio problem serves as a structured nonconvex benchmark for comparing the BSOS and BSOS–SPLD hierarchies.

\begin{table}[t]
\caption{Comparison of BSOS and BSOS-SPLD on the synthetic SPLD
portfolio problem for different dimensions $n$ and relaxation orders $k$ ($\lambda=0.02,$ $\eta=0.25$, $p=4$, $\theta=0.2$ in all tests).}
\label{table3}
\centering{\footnotesize
\begin{tabular}{ccccccccccc}
\toprule
\multicolumn{1}{l}{\textbf{Problem}} & \multicolumn{5}{l}{\textbf{BSOS}}  & \multicolumn{4}{l}{\textbf{BSOS-SPLD}} \\
\cmidrule(rl){2-6} \cmidrule(rl){7-11}
$n$  & $(d,k)$ & $\verb"Opt"$        & $\verb"Time"$ & $\verb"ms"$ & \verb"rnk" & $(d_0,r,k)$ & $\verb"Opt"$      & $\verb"Time"$ & $\verb"max_ms"$ & \verb"max_rnk" \\
\toprule
$6$  & $4,1$   & $-1.4159{\rm e}-01$ & $2.8$         & $210$       & $2$        & $27,2,1$ & $-1.4159{\rm e}-01$  & $0.1$         & $28$            & $2$             \\
     & $4,2$   & $-1.1321{\rm e}-01$ & $4.5$         & $210$       & $1$        & $27,2,2$ & $-1.1321{\rm e}-01$  & $0.2$         & $28$            & $1$             \\
\midrule
$8$  & $4,1$   & $-1.1603{\rm e}-01$ & $55.5$        & $495$       & $2$        & $44,2,1$ & $-1.1603{\rm e}-01$  & $0.2$         & $45$            & $2$             \\
     & $4,2$   & $-9.9594{\rm e}-02$ & $81.9$        & $495$       & $1$        & $44,2,2$ & $-9.9593{\rm e}-02$  & $0.4$         & $45$            & $1$             \\
\midrule
$10$ & $4,1$   & $-$                 & $-$           & $1001$      & $-$        & $65,2,1$ & $-8.5904{\rm e}-02$  & $0.5$         & $66$            & $1$             \\
\midrule
$15$ & $4,1$   & $-$                 & $-$           & $3876$      & $-$        & $135,2,1$ & $-7.8386{\rm e}-02$ & $6.4$         & $136$           & $1$             \\
\midrule
$20$ & $4,1$   & $-$                 & $-$           & $10626$     & $-$        & $230,2,1$ & $-1.4747{\rm e}-01$ & $56.5$        & $231$           & $2$             \\
     & $4,2$   & $-$                 & $-$           & $10626$      & $-$        & $230,2,2$ & $-7.3522{\rm e}-02$ & $76.5$        & $231$           & $1$             \\
\bottomrule
\end{tabular}}
\end{table}

Table~\ref{table3} reports the performance of BSOS and BSOS-SPLD on the portfolio problem \eqref{eq:spld-portfolio} for several asset dimensions $n$ and relaxation orders $k$.
For $n=6,8,$ the two hierarchies yield essentially the same optimal values at each relaxation order, but BSOS-SPLD is much cheaper:
it uses considerably smaller conic blocks (e.g., \texttt{ms} $=28$ vs.\ $210$ for $n=6$) and is between one and two orders of magnitude faster in CPU time.
For larger instances ($n\ge 10$) the standard BSOS relaxations become numerically intractable:
the moment matrices reach size $1001$ and the solver fails within the time/memory limits.
In contrast, BSOS-SPLD still solves all cases, with moderate matrix sizes (at most $231$) and reasonable computing times, while its bounds improve as $k$ increases.
Overall, the table illustrates that exploiting the SPLD structure improves scalability while preserving the tightness of the resulting bounds.

\begin{table}[t]
\centering
\caption{Characteristics of the optimal SPLD portfolio for $n=20$ under different values of $\lambda$ ($\eta=0.25$, $p=4$, $\theta=0.2$; solutions computed by BSOS-SPLD).}
\label{table4}
\centering{\footnotesize
\begin{tabular}{cccccccc}
\toprule
$\lambda$ & $\boldsymbol\mu^T \bx$         & $\bx^T \Sigma_0 \bx$ & $\bx^T Q \bx$ & $N_{\mathrm{eff}}$ & $\max_i x_i$ & $\#\{i:x_i\ge0.05\}$ & $\#\{i:x_i\ge0.01\}$  \\
\midrule
$0.02$    & -5.5974{\rm e}-05 & 4.2606{\rm e}-07 & $-0.1210$ & $2.0171$           & $0.6961$ & $2$  & $7$   \\
$0.2$     & -5.2215{\rm e}-05 & 3.5955{\rm e}-07 & $-0.1021$ & $2.3902$           & $0.6340$ & $3$  & $10$   \\
$2.0$     & -4.3953{\rm e}-05 & 2.4401{\rm e}-07 & $-0.0677$ & $3.5271$           & $0.5037$ & $4$  & $17$   \\
\bottomrule
\end{tabular}}
\end{table}

Table~\ref{table4} reports, for different values of the diversification parameter $\lambda$, the following portfolio statistics:
$\boldsymbol\mu^T \bx$ (expected excess return), $\bx^T\Sigma_0 \bx$ (variance under the shrunk covariance matrix $\Sigma_0$), $\bx^T Q \bx$ (value of the indefinite quadratic risk term), $N_{\mathrm{eff}} = 1/\sum_{i=1}^n x_i^2$ (effective
number of assets), $\max_i x_i$ (largest portfolio weight), and $\#\{i : x_i \ge 0.05\}$, $\#\{i : x_i \ge 0.01\}$ (numbers of assets whose weights are at least $5\%$ and $1\%$, respectively).

Table~\ref{table4} illustrates that the diversification parameter $\lambda$ effectively controls the concentration of the optimal portfolio.
As $\lambda$ increases, the high-degree separable penalty $\lambda\sum_i x_i^{2p}$ grows stronger, and the optimal portfolio shifts towards more diversified allocations:
$N_{\mathrm{eff}}$ increases, $\max_i x_i$ decreases, and the numbers of assets with weights at least $5\%$ and $1\%$ grow.
This confirms that the penalty successfully shifts the portfolio from a highly concentrated configuration towards a more balanced allocation.
In this setting, $\lambda$ thus serves as a diversification parameter, moderating the influence of the nonconvex quadratic term.

\section{A class of convex SPLD polynomial optimization problems}\label{sect:5}

\subsection{Convex polynomials and SPLD polynomials}

In this subsection, we study convex polynomials with SPLD structure and derive a key result on the degree relationship for convex SPLD polynomials.

\begin{lemma}{\rm \cite{Bank1988,Belousov2002}}\label{lemma1}
Let $f\colon \mathbb{R}^n \to \mathbb{R}$ be a convex polynomial of degree $d,$ and let $f_d$ denote the homogeneous polynomial of degree $d$ in $f.$
Then $f_d$ is convex.
\end{lemma}

\begin{lemma}\label{lemma2}
Let $p_d\colon \R^n\to\R$ be a homogeneous polynomial of degree $d\ge2$ consisting only of non-separable monomials. Then $p_d$ is not convex.
\end{lemma}

\begin{proof}
Assume to the contrary that $p_d$ is convex.
Since $p_d$ is homogeneous of degree $d\geq2$, we have $p_d(\bze)=0$ and $\nabla p_d(\bze)=\bze$.
Then, by convexity, for every $\bx\in\R^n$,
\begin{equation*}
p_d(\bx)\geq p_d(\bze)+\langle \nabla p_d(\bze), \bx-\bze\rangle=0,
\end{equation*}
and so, $p_d(\bx)\geq0$ for all $\bx\in\R^n$.
In particular, the set
\begin{equation*}
L:=\{\bx\in\R^n \mid p_d(\bx)=0\}=\{\bx\in\R^n \mid p_d(\bx)\leq0\}
\end{equation*}
is convex.

Observe that, by homogeneity, if $\bx\in L$ and $\lambda\in\R$, then $p_d(\lambda \bx)=\lambda^d p_d(\bx)=0,$ i.e., $\lambda \bx\in L$.
Moreover, if $\bx,\by\in L,$ then, by convexity of $L,$ we have $(\bx+\by)/2\in L.$
Hence, by homogeneity, we have
\begin{equation*}
\bx+\by=2\cdot\frac{\bx+\by}{2}\in L.
\end{equation*}
Thus, $L$ is a linear subspace of $\R^n.$

Next, we show that $\be_j\in L$ for all $j=1,\ldots,n,$ where each $\be_j$ is the $j$-th standard basis vector in $\R^n.$
Fix $j\in\{1,\ldots,n\}$.
Since every monomial of $p_d$ involves at least two distinct variables, each monomial vanishes at $\be_j$, and so, we have $p_d(\be_j)=0$.
Thus, $\be_j\in L$ for all $j=1,\ldots,n$.
This implies that $L=\R^n$ since $L$ is a linear subspace of $\R^n.$
Hence, $p_d(\bx)=0$ for all $\bx\in\R^n$, which contradicts the assumption that $\deg p_d = d \ge 2$.
Thus, $p_d$ is not convex.
\end{proof}

\begin{theorem}\label{thm3}
Let $f\in\R[\bx]_d$ be a convex polynomial such that $f(\bx) := s(\bx)+p(\bx)$ for all $\bx\in\R^n,$ where $s$ is a separable polynomial on $\R^n$ and $p$ is a polynomial on $\R^n$ consisting only of non-separable monomials (i.e.$,$ monomials involving at least two variables).
Then$,$
\begin{equation*}
 \deg(s) \geq \deg(p).
\end{equation*}
\end{theorem}

\begin{proof}
Let $f_d$ be the homogeneous term of degree $d$ in $f.$
Since $f$ is convex, it follows from Lemma~\ref{lemma1} that $f_d$ is also convex.

Assume to the contrary that $\deg(s)<\deg(p).$
Then $f_d=p_d,$ where $p_d$ is the homogeneous term of degree $d$ in $p$.
By assumption, $p_d$ consists only of non-separable monomials.
Hence, by Lemma~\ref{lemma2}, $p_d$ is not convex.
This contradicts the fact that $f_d$ is convex.
This completes the proof.
\end{proof}

\begin{remark}\label{rmk1}
{\rm Theorem~\ref{thm3} shows that convex polynomials naturally allow a decomposition $f=s+p,$ where $s$ is separable and $p$ contains only non-separable monomials with $\deg(s)\geq\deg(p)$.
Although this does not always satisfy the strict inequality required by the SPLD structure, it is still reasonable to expect that convex polynomials can often be expressed in SPLD form.
In particular, if $f$ is a convex polynomial of degree $d$ and all of its non-separable terms (i.e., monomials involving at least two variables) have degree strictly less than $d$, then $f$  must be an SPLD polynomial.
}\end{remark}

It is well known that a nonnegative SOS-convex polynomial is SOS (\cite{Helton2010}), whereas there exist nonnegative convex polynomials that are not sum-of-squares (see \cite{Saunderson2023} for the first constructive example).
In \cite{Ahmadi2022}, it is shown that a nonnegative convex SPQ polynomial can be rewritten as the sum of a nonnegative separable polynomial and a nonnegative quadratic polynomial.
Along these lines, we propose the following result, which shows that a nonnegative polynomial that is a separable convex polynomial plus an SOS-convex polynomial can be written as the sum of a nonnegative separable polynomial and a sum-of-squares polynomial.

\begin{lemma}\label{lemma4}
Let $f\colon \R^n\to\R$ be an SOS-convex polynomial.
Assume that $f$ can be written as $f:=s+p,$ where $s$ is a separable convex polynomial on $\R^n,$
i.e.$,$ $s:=\sum_{j=1}^nu_j$ for some convex univariate polynomials $u_j$ with degree at most $2d_j$ in $x_j,$ $j=1,\ldots,n,$ and $p$ is an SOS-convex polynomial on $\R^n$ with degree at most $2r.$
If $f(\bx)\geq0$ for all $\bx\in \R^n,$ then there exist $\sigma\in\Sigma[\bx]_{r}$ and $\sigma_j\in\Sigma[x_j]_{d_j},$ $j=1,\ldots,n,$ such that
\begin{align*}
f = \sigma+\sum_{j=1}^n \sigma_j.
\end{align*}
\end{lemma}

\begin{proof}
Let $f$ be a nonnegative convex polynomial, it then follows from \cite[Theorem~3]{Belousov2002} that there exists $\overline \bx\in\R^n$ such that $f(\overline \bx)=\min_{\bx\in\R^n}f(\bx),$ and so, $\nabla f(\overline \bx)=0.$
Without loss of generality, we may assume that $f(\overline \bx) = 0.$
So, by assumption that $f = s+p,$ we have
\begin{align*}
f(\bx)=p(\bx)-p(\overline \bx)-\nabla p(\overline \bx)^T(\bx-\overline \bx)+\sum_{j=1}^n\left\{u_j(x_j)-u_j(\overline x_j)-u_j'(\overline x_j)(x_j-\overline x_j)\right\}
\end{align*}
for all $\bx\in\R^n.$
Since $p$ is an SOS-convex polynomial, $p(\bx)-p(\overline \bx)-\nabla p(\overline \bx)^T(\bx-\overline \bx)$ is SOS with degree at most $2r.$
Moreover, since each $u_j$ is convex, $u_j(x_j)-u_j(\overline x_j)-u_j'(\overline x_j)(x_j-\overline x_j)$ is nonnegative, and so, it is SOS with degree at most $2d_j$ \cite{Hilbert1888}.
Thus, the desired result follows.
\end{proof}

For an SOS-convex polynomial $f$ of degree $2d$, its Hessian $H(\bx)=\nabla^2 f(\bx)$ is, by definition, an SOS-matrix (of degree at most $2d-2$).
Hence there exists a Gram matrix $Q \succeq 0$ such that
\begin{equation*}
H(\bx) = \left(I_n \otimes \lceil\bx\rceil_{d-1}\right)^T Q \left(I_n \otimes \lceil\bx\rceil_{d-1}\right),
\end{equation*}
where $I_n$ is the $n\times n$ identity and $\otimes$ is the Kronecker product (see, e.g., \cite{Ahmadi2012}).

In order to exploit the additional separable structure required in Lemma~\ref{lemma4}, we look for a more structured Gram representation of $H$.
The structured assumption in Lemma~\ref{lemma4} (existence of a decomposition $f=s+p$ with a separable convex polynomial $s$ and an SOS-convex polynomial $p$) can be verified by checking whether there exist matrices $Q_j\succeq 0,$ $j=0,1,\ldots,n,$ such that
\begin{equation}\label{Hessian}
H(\bx) = (I_n\otimes \lceil\bx\rceil_{r-1})^T Q_0 (I_n\otimes \lceil\bx\rceil_{r-1})+\sum_{j=1}^n (\lceil x_j\rceil_{d_j-1} \be_j^T )^T Q_j (\lceil x_j\rceil_{d_j-1}\be_j^T),
\end{equation}
where each $\be_j$ is the $j$-th standard basis vector in $\R^n.$
Equating coefficients of the monomials in $\bx$ in the matrix identity \eqref{Hessian} yields linear equations in the entries of $Q_j$, $j=0,1,\ldots,n$, together with the semidefinite constraints $Q_j\succeq0$, $j=0,1,\ldots,n$.
Hence, verifying the assumption of Lemma~\ref{lemma4} reduces to an SDP feasibility problem in these Gram matrices.
Moreover, this structured Gram representation can effectively exploit the term sparsity of SPLD polynomials to further reduce the size of the semidefinite matrices.\footnote{Term sparsity refers to the property of a polynomial having few nonzero monomial terms relative to the total number of possible monomials up to a given degree.}

More specifically, let $f\in\Sigma[\bx]_d,$ i.e., $\sum_{\ell=1}^{q}f_{\ell}^2$ for some $f_\ell\in\R[\bx]_d.$
Then by \cite[Theorem~1]{Reznick1978}, we have $\mathcal{N}(f_\ell)\subseteq \frac{1}{2}\mathcal{N}(f)$ for all $\ell,$ where $\mathcal{N}(f)$ denotes Newton polytope defined as $\mathcal{N}(f):= {\rm conv}\{\ba\in\N^n_{2d} : f_\ba\neq0\}.$
If $f$ is not only an SOS polynomial, but also an SPLD polynomial, then each $f_\ell$ may involve significantly fewer monomials than the full monomial basis $(\bx^\ba)_{\ba \in \N_d^n}$ in $\R[\bx]_{2d}.$
As a result, it is possible to reduce the size $s(n,d)$ of a semidefinite matrix corresponding to an SOS decomposition of $f$.
Furthermore, if $f$ has a decomposition as in Lemma~\ref{lemma4}, then the maximum size of a semidefinite matrix can be further reduced (compared to \cite[Theorem~1]{Reznick1978}); see the forthcoming Example~\ref{ex1} for more detailed explanation.

\subsection{Exact SOS relaxations for SOS-convex SPLD polynomial optimization}

In this subsection, we focus on a more tractable subclass of convex SPLD polynomial optimization problems. Namely, we consider the following convex optimization problem:
\begin{align}\label{CP}
\inf\limits_{\bx\in\R^n} \ \left\{f(\bx) \colon \  g_i(\bx)\leq 0, \ i=1,\ldots,m\right\}, \tag{${\rm CP}$}
\end{align}
where $f,g_i\colon\R^n\to\R,$ $i=1,\ldots,m,$ are convex polynomials such that
\begin{align*}
f(\bx):=s_0(\bx)+p_0(\bx) \ \textrm{ and } \ g_i(\bx):=s_i(\bx)+p_i(\bx), \ i=1,\ldots,m,
\end{align*}
for all $\bx\in\R^n,$ and for each $i = 0, 1, \ldots, m$, $s_i\colon\R^n\to\R$ is a separable convex polynomial, i.e.,
\begin{align*}
s_i(\bx):=\sum_{j=1}^nu_i^j(x_j),
\end{align*}
namely, each $u_i^j(x_j)$ is a convex polynomial in $x_j$ with degree at most $2d_j$ and each $p_i\colon\R^n\to\R$ is an SOS-convex polynomial (it may include separable polynomial terms) with degree at most $2r.$
We denote the feasible set of the problem \eqref{CP} as ${\rm\bf{F}}:=\{\bx\in \R^n \colon g_i(\bx)\leq 0, \ i=1,\ldots,m\}.$
Note that the feasible set ${\rm\bf{F}}$ is no longer required to be compact.

Let $d_0:=\max_{j=1,\ldots,n}d_j.$
From now on, we assume that (i) the objective function of the problem~\eqref{CP} is bounded from below on ${\bf F};$ and (ii) $d_0 > r,$ which guarantees that the considered polynomials satisfy the SPLD structure.

\begin{assumption}\label{assumption3}
There exists $\widehat \bx\in {\rm\bf{F}}$ such that $g_i(\widehat \bx)<0$ for all $i=1,\ldots,m.$
\end{assumption}

Now, we formulate the Lagrangian dual problem of the problem \eqref{CP} as follows:
\begin{align}\label{LD}
\sup\limits_{\substack{\mu\in\R,\lambda_i\geq0 }}\left\{ \mu \colon f+\sum_{i=1}^m\lambda_ig_i-\mu\geq0\right\}.\tag{${\rm LD}$}
\end{align}
It is well-known that strong duality holds between \eqref{CP} and \eqref{LD} if Assumption~\ref{assumption3} is satisfied, i.e., ${\rm val}\eqref{CP}={\rm val}\eqref{LD}$; see, e.g., \cite[Section~5.3]{Boyd2004}.

Observe that $f$ and $g_i,$ $i=1,\ldots,m,$ are also SOS-convex polynomials; moreover, by \cite[Theorem~3.3]{Lasserre2009}, the problem \eqref{CP} exhibits an exact sum-of-squares relaxation under some suitable condition (e.g., Assumption~\ref{assumption3}), i.e., ${\rm val}\eqref{CP}={\rm val}\eqref{LDSOS},$ where \eqref{LDSOS} is given by
\begin{align}\label{LDSOS}
\sup\limits_{\substack{\mu\in\R,\lambda_i\geq0,\sigma\in\Sigma[\bx]_{d_0} }}\left\{ \mu \colon f+\sum_{i=1}^m\lambda_ig_i-\mu=\sigma\right\},\tag{${\rm D}_{\rm sos}$}
\end{align}
which can be represented by an SDP problem.

For the dual problem \eqref{LD}, compared to the SOS relaxation \eqref{LDSOS}, we now consider the following relaxation problem:
\begin{align}\label{D}
\sup\limits_{\substack{\mu\in\R,\lambda_i\geq0,\\ \sigma\in\Sigma[\bx]_{r}, \sigma_j\in\Sigma[x_j]_{d_j} }}\left\{ \mu \colon f+\sum_{i=1}^m\lambda_ig_i-\mu=\sigma+\sum_{j=1}^n\sigma_j\right\}.\tag{${\rm D_{\rm sos}^{\rm SPLD}}$}
\end{align}
It is known that, since $\sigma+\sum_{j=1}^n\sigma_j\in \Sigma[\bx]_{d_0}$ for all $\sigma\in\Sigma[\bx]_{r}$ and all $\sigma_j\in\Sigma[x_j]_{d_j},$ we have
\begin{align*}
{\rm val}\eqref{D}\leq{\rm val}\eqref{LDSOS}.
\end{align*}

\begin{remark}{\rm
Note that the sizes of the semidefinite constraints in the two problems \eqref{LDSOS} and \eqref{D} differ significantly.
The semidefinite constraint in the problem \eqref{LDSOS} is of size $s(n,d_0),$ whereas the size of the largest semidefinite constraint in the problem \eqref{D} is
\begin{align*}
\max\{d_0+1,s(n,r)\},
\end{align*}
i.e., the size of the problem \eqref{D} can be much smaller than the problem \eqref{LDSOS} if $d_0\gg r.$
}\end{remark}

Similar to the approach in Section~\ref{subsection3.1} for the problem \eqref{Dkdr}, we can reformulate the problem \eqref{D} as the following SDP problem:
\begin{align*}
\sup\limits_{\substack{\mu\in\R,\, \lambda_i\geq0, \\ X\in \mathcal{S}^{{s(n,r)}}_+,\, X_j\in \mathcal{S}^{d_j+1}_+ }}\left\{ \mu :  \left(f+\sum_{i=1}^m\lambda_ig_i-\mu\right)_\ba=\langle B_\ba, X\rangle+\sum_{j=1}^n\langle B_{\ba}^j, X_j\rangle, \ \forall \ba\in \N^n_{2d_0}\right\},
\end{align*}
where the matrices $B_\ba\in \mathcal{S}^{s(n,r)}$  and $B_{\ba}^j\in\mathcal{S}^{d_j+1},$ $j=1,\ldots,n,$ are defined as in \eqref{Ba1} for all $\ba\in \N^n_{2d_0}.$
Then its Lagrangian dual problem can be formulated as follows:
\begin{align}\label{SDP}
\inf\limits_{\substack{\by\in\R^{s(n,2d_0)}}} \quad& \sum_{\ba\in\N^n_{2d_0}}f_\ba y_\ba \tag{${\rm \widetilde D}_{\rm sos}^{\rm SPLD}$} \\
{\rm s.t. }  \qquad\  &  \sum_{\ba\in\N^n_{2d_0}}(g_i)_\ba y_\ba\leq0, \  i=1,\ldots,m, \nonumber \\
&\sum_{\ba\in\N^n_{2d_0}}y_\ba (\lceil x_j\rceil_{d_j}\lceil x_j\rceil_{d_j}^T)_\ba\succeq0, \ j=1,\ldots,n, \nonumber \\
&\sum_{\ba\in\N^n_{2d_0}}y_\ba (\lceil \bx\rceil_r\lceil \bx\rceil_r^T)_\ba\succeq0,  \nonumber \\
&   y_{\bze}=1.\nonumber
\end{align}

The following result shows that the strong duality holds for the problems \eqref{CP}, \eqref{D}, and \eqref{SDP}.
\begin{theorem}
For the problem \eqref{CP}$,$ assume that $p_i,$ $i=0,1,\ldots,m,$ are SOS-convex polynomials.
If {\rm Assumption~\ref{assumption3}} holds$,$ then we have
\begin{align*}
{\rm val}\eqref{CP}={\rm val}\eqref{D}={\rm val}\eqref{SDP}
\end{align*}
and the problem \eqref{D} has an optimal solution.
\end{theorem}

\begin{proof}
It suffices to show that ${\rm val}\eqref{CP}\leq{\rm val}\eqref{D}$ since the rest of the proof is similar to that of Theorem~\ref{thm1}.
Under Assumption~\ref{assumption3}, by strong duality for the convex optimization problems (see, e.g., \cite{Boyd2004}), we have ${\rm val}\eqref{CP}={\rm val}\eqref{LD}.$

Let $(\overline\mu,\overline\lambda_1,\ldots,\overline\lambda_m)\in\R\times\R^m_+$ be a maximizer of the problem \eqref{LD}.
Then, by structure of $f$ and each $g_i$, it yields that $f(\bx)+\sum_{i=1}^m\overline\lambda_ig_i(\bx)-\overline\mu$ is an SOS-convex SPLD polynomial.
Moreover, since $f+\sum_{i=1}^m\overline\lambda_ig_i-\overline\mu$ is a nonnegative polynomial in $\bx,$ by Lemma~\ref{lemma4}, there exist $\overline\sigma\in\Sigma[\bx]_{r}$ and $\overline\sigma_j\in\Sigma[x_j]_{d_j},$ $j=1,\ldots,n,$ such that
\begin{align*}
f+\sum_{i=1}^m\overline\lambda_ig_i-\overline\mu=\overline\sigma+\sum_{j=1}^n\overline\sigma_j,
\end{align*}
and so, $(\overline\mu,(\overline\lambda_i),\overline\sigma,(\overline\sigma_j))$ is feasible for the problem \eqref{D}.
Hence, ${\rm val}\eqref{CP}=\overline\mu\leq {\rm val}\eqref{D}.$
\end{proof}

In what follows, we provide a way to recover an optimal solution to the problem \eqref{CP} from a solution to the problem \eqref{SDP} whenever $p_i,$ $i=0,1,\ldots,m,$ are SOS-convex polynomials in the problem \eqref{CP}.
\begin{theorem}[Recovery of a solution]\label{thm5}
For the problem \eqref{CP}$,$ assume that $p_i,$ $i=0,1,\ldots,m,$ are SOS-convex polynomials.
Suppose that {\rm Assumption~\ref{assumption3}} holds.
If $\overline \by\in \R^{s(n,2d_0)}$ is an optimal solution to the problem \eqref{SDP}$,$ then $\overline \bx:=(\overline y_\ba)_{|\ba|=1}\in\R^n$ is an optimal solution to the problem \eqref{CP}.
\end{theorem}
\begin{proof}
This result can be derived by a similar argument to the proof of Theorem~3.3 (b) in \cite{Lasserre2009}.
\end{proof}

We finish this part by presenting the following simple example, which illustrates how an optimal solution to the problem \eqref{CP} can be found according to Theorem~\ref{thm5}.
\begin{example}\label{ex1}{\rm
Consider the following convex optimization problem:
\begin{align}\label{CP1}
\inf\limits_{(x_1,x_2)\in\R^2} \quad & f(x_1,x_2):=x_1^8-x_1^6+x_1^4+x_1^2x_2^2+x_2^4+x_1^2x_2+x_1x_2^2+x_1^2+x_2^2  \tag{${\rm CP}_1$} \\
 {\rm s.t.\ } \quad \ \ \,  & g_1(x_1,x_2):=x_1^2+x_2^2-1\leq 0.  \nonumber
\end{align}
Observe that both $f$ and $g_1$ are (SOS-)convex polynomials.
By the KKT condition for the problem~\eqref{CP1}, one can verify that $(\overline x_1,\overline x_2)=(0,0)$ is the unique optimal solution.

Now, we formulate the sum-of-squares relaxation problem of \eqref{CP1}:
\begin{align}\label{LDSOS1}
\sup\limits_{\substack{\mu\in\R,\lambda_1\geq0 }}\left\{ \mu \colon f+\lambda_1g_1-\mu\in \Sigma[(x_1,x_2)]_{4}\right\}.\tag{${\rm D}_{\rm sos}^1$}
\end{align}
Note that the semidefinite matrix associated with the sum-of-squares decomposition of the problem \eqref{LDSOS1} has size $s(2,4)=15.$
Let $(\mu,\lambda_1)$ be a feasible solution of the problem \eqref{LDSOS1} and define $\phi:=f+\lambda_1g_1-\mu.$
Then, by \cite[Theorem~1]{Reznick1978}, we check that
\begin{equation*}
\frac{1}{2}\mathcal{N}(\phi)\cap \N^2=\{(0,0),(1,0),(0,1),(2,0), (1,1), (0,2),(3,0),(2,1),(4,0)\}.
\end{equation*}
With this result, we formulate the dual problem of the problem \eqref{LDSOS1} as follows:
\begin{align}\label{LSDP1}
\inf\limits_{\by\in\R^{25}}\quad  &y_{80}-y_{60}+y_{40}+y_{22}+y_{04}+y_{21}+y_{12}+y_{20}+y_{02}  \tag{${\rm SDP}_1$} \\
 {\rm s.t.\ } \quad  & y_{20}+y_{02}-1\leq 0,   \nonumber\\
&\left(\begin{array}{ccccccccc}
y_{00} & y_{10} & y_{01} & y_{20} & y_{11} & y_{02} & y_{30} & y_{21} & y_{40} \\
y_{10} & y_{20} & y_{11} & y_{30} & y_{21} & y_{12} & y_{40} & y_{31} & y_{50} \\
y_{01} & y_{11} & y_{02} & y_{21} & y_{12} & y_{03} & y_{31} & y_{22} & y_{41} \\
y_{20} & y_{30} & y_{21} & y_{40} & y_{31} & y_{22} & y_{50} & y_{41} & y_{60} \\
y_{11} & y_{21} & y_{12} & y_{31} & y_{22} & y_{13} & y_{41} & y_{32} & y_{51} \\
y_{02} & y_{12} & y_{03} & y_{22} & y_{13} & y_{04} & y_{32} & y_{23} & y_{42} \\
y_{30} & y_{40} & y_{31} & y_{50} & y_{41} & y_{32} & y_{60} & y_{51} & y_{70} \\
y_{21} & y_{31} & y_{22} & y_{41} & y_{32} & y_{23} & y_{51} & y_{42} & y_{61} \\
y_{40} & y_{50} & y_{41} & y_{60} & y_{51} & y_{42} & y_{70} & y_{61} & y_{80} \\
\end{array}\right)\succeq0, \nonumber \\
& y_{00}=1.  \nonumber
\end{align}
Observe that the size of the semidefinite matrix in the problem \eqref{LSDP1} is $9,$ which is smaller than $15.$
Solving the semidefinite programming problem \eqref{LSDP1} using CVX \cite{Grant2011}, we obtain the optimal value ${\rm val}\eqref{LSDP1} = 0$ and an optimal solution
{\footnotesize
\begin{align*}
\overline \by
&=(\overline y_{00}, \overline y_{10}, \overline y_{01}, \overline y_{20}, \overline y_{11}, \overline y_{02}, \overline y_{30},\overline y_{21},\overline y_{12},\overline y_{03}, \overline y_{40}, \overline y_{31}, \overline y_{22}, \overline y_{13}, \overline y_{04}, \overline y_{50},\overline y_{41},\overline y_{32},\overline y_{23}, \overline y_{60}, \overline y_{51}, \overline y_{42}, \overline y_{70},\overline y_{61}, \overline y_{80})\\
&=(1,0,0,0,0,0,0,0,0,0,0,0,0,0,0,0,0,0,0,0,0,0,0,0,0).
\end{align*}
}

Let us now observe that $f$ can be rewritten in the following form:
\begin{equation*}
f(x_1,x_2):=\left\{x_1^8-x_1^6+\textstyle{\frac{1}{3}}x_1^4\right\}+\left\{\textstyle{\frac{2}{3}}x_1^4+x_1^2x_2^2+x_2^4+x_1^2x_2+x_1x_2^2+x_1^2+x_2^2\right\}.
\end{equation*}
Let $u_0^1(x_1):=x_1^8-x_1^6+\frac{1}{3}x_1^4,$ $p_0(x_1,x_2):=\frac{2}{3}x_1^4+x_1^2x_2^2+x_2^4+x_1^2x_2+x_1x_2^2+x_1^2+x_2^2,$ and let $p_1(x_1,x_2)=x_1^2+x_2^2-1.$
Then $f(x_1,x_2)=u_0^1(x_1)+p_0(x_1,x_2)$ and $g_1(x_1,x_2)=p_1(x_1,x_2).$
Moreover, we can verify that $u_0^1,$ $p_0,$ and $p_1$ are SOS-convex polynomials.
Now we consider the following SDP problem for \eqref{CP1}:
\begin{align}\label{SDP1}
\inf\limits_{\by\in\R^{19}} \quad & y_{80}-y_{60}+y_{40}+y_{22}+y_{04}+y_{21}+y_{12}+y_{20}+y_{02}  \tag{${\rm SDP}_2$} \\
 {\rm s.t.\ } \quad  & y_{20}+y_{02}-1\leq 0,   \nonumber\\
&\left(\begin{array}{ccccc}
 y_{00} & y_{10} & y_{20} & y_{30} & y_{40} \\
 y_{10} & y_{20} & y_{30} & y_{40} & y_{50} \\
 y_{20} & y_{30} & y_{40} & y_{50} & y_{60} \\
 y_{30} & y_{40} & y_{50} & y_{60} & y_{70} \\
 y_{40} & y_{50} & y_{60} & y_{70} & y_{80} \\
\end{array}\right)\succeq0, \nonumber \\
&\left(\begin{array}{ccccccccc}
 y_{00} & y_{10} & y_{01} & y_{20} & y_{11} & y_{02} \\
 y_{10} & y_{20} & y_{11} & y_{30} & y_{21} & y_{12} \\
 y_{01} & y_{11} & y_{02} & y_{21} & y_{12} & y_{03} \\
 y_{20} & y_{30} & y_{21} & y_{40} & y_{31} & y_{22} \\
 y_{11} & y_{21} & y_{12} & y_{31} & y_{22} & y_{13} \\
 y_{02} & y_{12} & y_{03} & y_{22} & y_{13} & y_{04} \\
\end{array}\right)\succeq0, \nonumber\\
&    y_{00}=1.  \nonumber
\end{align}
Note here that the maximum size of the matrices in the problem~\eqref{SDP1} is $6,$ which is less than the size of the semidefinite matrix in the problem \eqref{LSDP1}.
Solving \eqref{SDP1} using CVX \cite{Grant2011}, we obtain the optimal value ${\rm val}\eqref{SDP1} = 0$ and an optimal solution
\begin{align*}
\overline \by
&=(\overline y_{00}, \overline y_{10}, \overline y_{01}, \overline y_{20}, \overline y_{11}, \overline y_{02}, \overline y_{30},\overline y_{21},\overline y_{12},\overline y_{03}, \overline y_{40}, \overline y_{31}, \overline y_{22}, \overline y_{13}, \overline y_{04}, \overline y_{50}, \overline y_{60}, \overline y_{70},\overline y_{80})\\
&=(1,0,0,0,0,0,0,0,0,0,0,0,0,0,0,0,0,0,0).
\end{align*}
So, $(\overline x_1,\overline x_2)=(0,0)=( \overline y_{10}, \overline y_{01}),$ and thus, Theorem~\ref{thm5} holds.
}\end{example}

\subsection{SOS-convex SPLD polynomial regression}\label{subsec:5}
Here, we consider the problem of {\em convex polynomial} regression.
Convex regression is a fundamental problem in machine learning and statistics:
it seeks a convex function $f$ that maps {\em feature vectors} to {\em response variables}.
The goal is to find an estimator for $f$ given noisy feature-response data.
The research history of addressing this problem dates back to the study of estimation of the convex least squares estimator by Hildreth \cite{Hildreth1954} (see also Holloway \cite{Holloway1979}).
In this subsection, we focus on the scenario where the dataset is large and the dimension is relatively small, and we aim to estimate $f$ using an SOS-convex SPLD polynomial.

Assume now that we have $m$ datapoints $(\bx_i, y_i),\ i = 1, \ldots, m,$ of an unknown convex function $f : \R^n \to \R.$
To extrapolate $f$ at new points, we want to find a (convex) polynomial $p: \R^n \to \R$ (of a given degree) that best explains the observations.
To this end, we consider the following sum of the absolute deviations between the
observed values $y_i$ and the predicted ones:
\begin{align*}
\min_{p} \sum_{i = 1}^m | p(\bx_i) - y_i|.
\end{align*}

In order to exploit the convexity of $f$, one sometimes has to impose a convexity constraint on the regressor $p(\bx).$ However, the convexity constraint makes the regression problem intractable.
Very recently, to overcome the intractability of a convexity constraint on $p(\bx),$ Curmei and Hall \cite{Curmei2025} impose an {\em SOS-convexity constraint} on $p(\bx)$ and solve the regression problem via semidefinite programming; moreover, Ahmadi et al. \cite[Section 6.2]{Ahmadi2022} approximate $f$ with a {\em convex SPQ polynomial}, which helps significantly with the scalability of the resulting regression problem.

We now consider the following SOS-convex SPLD regression problem, which covers the convex SPQ regression problem,
\begin{align}\label{regression1}
\min_{p = s+l}\quad & \sum_{i = 1}^m | p(\bx_i) - y_i|\\
{\textrm{s.t.}}\quad\  & s(\bx) = \sum_{j = 1}^n s_j(x_j),\ s_j \in \R[x_j]_{2d_0},\ j = 1, \ldots, n, \  l \in \R[\bx]_{2r}, \notag \\
&H(\bx)=L(\bx)+{\rm diag}\left(u_1(x_1),\ldots,u_n(x_n)\right), \quad \forall \bx\in\R^n, \notag\\
&u_j \in \Sigma[x_j]_{d_0-1},\ j = 1, \ldots, n, \notag\\
&L(\bx) \textrm{ is an $n \times n$ SOS matrix of degree $2r - 2$}, \notag
\end{align}
where the decision variables are the coefficients of $p(\bx),$ the degrees $2d_0$ and $2r$ are fixed, $H(\bx)$ are the Hessian matrix of $p(\bx)$, and ${\rm diag}\left(u_1(x_1),\ldots,u_n(x_n)\right)$ is the diagonal matrix with the vector $\left(u_1(x_1),\ldots,u_n(x_n)\right)^T$ on its diagonal.

For our numerical experiment, we consider the following class of convex functions\footnote{This class of convex functions in \eqref{logfunction} is motivated by Ahmadi et al. \cite[Section 6.2]{Ahmadi2022}, where the quadratic term is modified by a quartic one.}:

\begin{align}\label{logfunction}
f_{\va^\ell,\vb^\ell}(\bx) = \log \left(\sum_{j = 1}^n a_j^\ell e^{b_j^\ell x_j} \right)+\left(\bx^T\left((\va^\ell) (\va^\ell)^T+I\right)\bx\right)^2+(\vb^\ell)^T(\bx),
\end{align}
where the entries of $\va^\ell, \vb^\ell \in \R^n$ are given uniformly and independently from the interval $[0, 2]$ and $[-1, 1],$ respectively, and $I$ is the $n \times n$ identity matrix.
We consider $100$ functions randomly chosen from the class of convex functions $f_{\va^\ell,\vb^\ell}$  in $6$ variables. 
For each instance $\ell = 1, \ldots, 100$, we have a training set of $m = 400$ random vectors $\bx_i \in \R^6$ drawn independently from the (multivariate) standard normal distribution.
The values $y_i$ are then computed as $y_i = f_{\va^\ell,\vb^\ell}(\bx_i)+\epsilon_i,$ where $\epsilon_i$ is again chosen independently from the standard normal distribution.
We focus on polynomials of degree $2d_0 = 6$ ($2r = 4$ for the lower-degree term of SPLD polynomials), and compare the performance of our problem~\eqref{regression1} with that of the following two regression models:
\begin{itemize}
\item SOS-convex polynomial regression problem \cite{Curmei2025},
\begin{align*}
\min_{p \in \R[\bx]_{2d_0}}\quad  \sum_{i = 1}^m | p(\bx_i) - y_i|\quad
{\textrm{s.t.}}\quad\ H(\bx) \textrm{ is an $n \times n$ sos-matrix of degree $2d_0 - 2$};
\end{align*}
\item convex SPQ polynomial regression problem \cite{Ahmadi2022},
\begin{align*}
\min_{p = s+q}\quad & \sum_{i = 1}^m | p(\bx_i) - y_i|\\
{\textrm{s.t.}}\quad\  & s(\bx) = \sum_{j = 1}^n s_j(x_j),\ s_j \in \R[x_j]_{2d_0},\ j = 1, \ldots, n, \  q \in \R[\bx]_{2}, \\
&H(\bx)=Q+{\rm diag}\left(u_1(x_1),\ldots,u_n(x_n)\right), \quad \forall \bx\in\R^n, \\
&Q \in \mathcal{S}^{n}_+, \ u_j \in \Sigma[x_j]_{d_0-1},\ j = 1, \ldots, n.
\end{align*}
\end{itemize}

\begin{figure}[t]
\centering
\begin{subfigure}{0.45\textwidth}
\centering
\includegraphics[width=\linewidth]{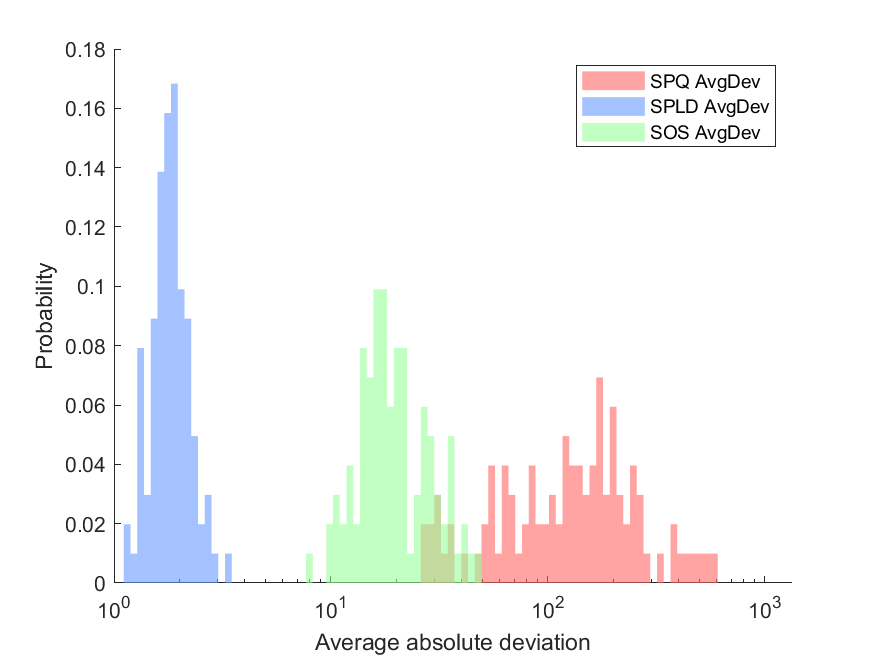}
\end{subfigure}
\hfill
\begin{subfigure}{0.45\textwidth}
\centering
\includegraphics[width=\linewidth]{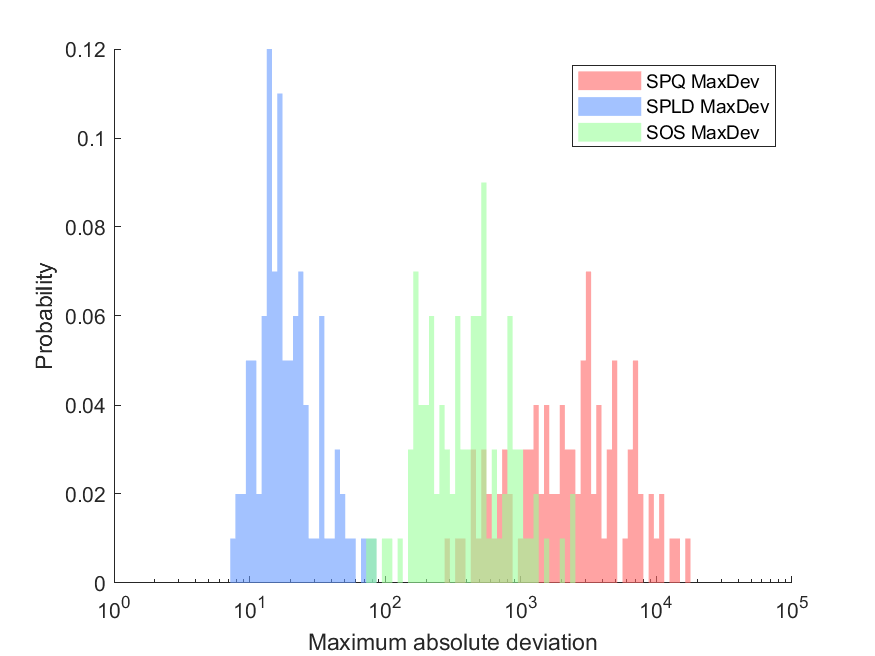}
\end{subfigure}
\caption{Log-scaled histograms of the average (left) and maximum (right) absolute deviation errors over 100 independent test runs, comparing three regression models: SOS-convex, convex SPQ, and SOS-convex SPLD. The $x$-axis represents the deviation values (in logarithmic scale), and the $y$-axis shows the normalized frequency (i.e., empirical probability) of observing those deviations across the test runs.}
\label{fig:1}
\end{figure}

Now, we solve the three regression problems using the training set and obtain an optimal polynomial $p^*$ in each case.
For each instance $\ell = 1, \ldots, 100$, we have a test set of $t = 100$ random vectors $\bx_i \in \R^6$ also drawn independently from the (multivariate) standard normal distribution.
By running 100 times, we then calculate the means of average absolute deviation error and of the maximum absolute deviation error over the test set:
\begin{align*}
\textrm{AvgDev} = \frac{1}{100} \sum_{\ell = 1}^{100} \textrm{AvgDev}_{\va^\ell,\vb^\ell} \quad \textrm{and} \quad
\textrm{MaxDev} = \frac{1}{100} \sum_{\ell = 1}^{100} \textrm{MaxDev}_{\va^\ell,\vb^\ell},
\end{align*}
where
\begin{align*}
\textrm{AvgDev}_{\va^\ell,\vb^\ell} = \frac{1}{t} \sum_{i=1}^{t} |f_{\va^\ell,\vb^\ell}(\bx_i) - p^*(\bx_i)| \quad \textrm{and} \quad
\textrm{MaxDev}_{\va^\ell,\vb^\ell} = \max_{1 \leq i \leq t} |f_{\va^\ell,\vb^\ell}(\bx_i) - p^*(\bx_i)|,
\end{align*}
respectively.

Figure~\ref{fig:1} shows the test error distributions of three regression models: SOS-convex, convex SPQ, and SOS-convex SPLD regression.
These results suggest that the SOS-convex SPLD model provides better accuracy and robustness than the other two models.
In Table~\ref{table-0}, we compare the test errors over $100$ instances: as shown, the SOS-convex SPLD polynomial regression achieves significantly smaller errors than both the SOS-convex and convex SPQ regression models.

\begin{table}[t]
\caption{Comparison of average absolute deviation error and maximum absolute deviation error of three regression problems.}\label{table-0}
\small
\begin{tabular}{cccc}
\hline
  & SOS-convex regression \cite{Curmei2025} & convex SPQ regression \cite{Ahmadi2022} & SOS-convex SPLD regression \\ \hline
\textrm{AvgDev} & 20.9509 & 160.3130 & 1.8449 \\ \hline
\textrm{MaxDev} & 517.4961 & 3.5283e+03 & 21.0255 \\ \hline
\end{tabular}
\end{table}

\section{Conclusions and further remarks}\label{sect:6}
In this paper, we have introduced a new class of structured polynomials, called SPLD polynomials, which extends the concept of SPQ polynomials \cite{Ahmadi2022}.
We have also proposed a variant of the bounded degree SOS hierarchy, called BSOS-SPLD, to solve a class of optimization problems involving SPLD polynomials.
The numerical experiments have illustrated that our approach can achieve better performance than the standard BSOS hierarchy \cite{Lasserre2017}.

On the other hand, as demonstrated with the test problem ({\verb"SPM"}), the standard BSOS hierarchy showed limited performance on SPLD problems involving high-degree polynomials with a small number of variables.
In contrast, our BSOS-SPLD hierarchy handled such cases effectively.
An important direction for future work is to address large-scale SPLD polynomial optimization (with high-degree and many variables) (cf. \cite{Weisser2018}).
We plan to investigate this in a forthcoming study.

\subsection*{Acknowledgments}
The authors would like to thank the associate editor and two anonymous referees for their careful reading with constructive comments and suggestions on an earlier version of this paper.

\subsection*{Funding}
Liguo Jiao was partially supported by the Chinese National Natural Science Foundation under grant numbers 12471478, 12371300, and was supported by the Education Department of Jilin Province (no. JJKH20241405KJ).
Jae Hyoung Lee was supported by the National Research Foundation of Korea (NRF) grant funded by the Korea government (MSIT) (NRF-2021R1C1C2004488).

\subsection*{Availability of data and materials}
The codes and datasets generated in this study are available from the corresponding author on reasonable request.

\subsection*{Conflict of interest}
The authors have no conflict of interest to declare that are relevant to the content of this article.

\section*{Appendix}

\begin{proof}[Proof of Theorem~\ref{thm1}]
Note first that, for all $k\in\N,$ the inequality ${\rm val}\eqref{Dkdr}\leq{\rm val}\eqref{LDkdr}$ always holds by the weak duality for semidefinite programming \cite{Vandenberghe1995}.
We now prove that ${\rm val}\eqref{LDkdr}\leq{\rm val}\eqref{P}$ for all $k\in\N.$
To see this, let $\bx$ be any feasible solution of the problem \eqref{P}, and let $\by:=(y_\ba)=(\bx^\ba)\in\R^{s(n,2\dd)}.$
Then we have $y_{\bze}=1$.
Moreover,  we see that
\begin{align*}
0\leq h_{\p,\q}(\bx)= \sum_{\ba\in\N^n_{2\dd}}(h_{\p,\q})_\ba \bx^\ba=\sum_{\ba\in\N^n_{2\dd}}(h_{\p,\q})_\ba y_\ba
\end{align*}
for all $(\p,\q)\in N_k.$
Observe that
\begin{align*}
\lceil x_j\rceil_{d_j}\lceil x_j\rceil_{d_j}^T&=\sum_{\ba\in\N^n_{2\dd}}\bx^\ba (\lceil x_j\rceil_{d_j}\lceil x_j\rceil_{d_j}^T)_\ba=\sum_{\ba\in\N^n_{2\dd}}y_\ba (\lceil x_j\rceil_{d_j}\lceil x_j\rceil_{d_j}^T)_\ba\succeq0, \ j=1,\ldots,n,\\
\lceil\bx\rceil_r\lceil\bx\rceil_r^T&=\sum_{\ba\in\N^n_{2\dd}}\bx^\ba (\lceil \bx\rceil_r\lceil \bx\rceil_r^T)_\ba=\sum_{\ba\in\N^n_{2\dd}}y_\ba (\lceil \bx\rceil_r\lceil \bx\rceil_r^T)_\ba\succeq0.
\end{align*}
Hence $\by\in\R^{s(n,2\dd)}$ is a feasible solution of the problem $\eqref{LDkdr}.$
So, it leads to the fact that
\begin{align*}
f_0(\bx)=\sum_{\ba\in\N^n_{2\dd}}(f_0)_\ba y_\ba \geq {\rm val}\eqref{LDkdr}.
\end{align*}
Since $\bx$ is arbitrary feasible for the problem \eqref{P}, we have that ${\rm val}\eqref{LDkdr}\leq {\rm val}\eqref{P},$ and so,
\begin{align}\label{thm1rel1}
{\rm val}\eqref{Dkdr} \leq {\rm val}\eqref{LDkdr} \leq {\rm val}\eqref{P} \quad \textrm{for all $k \in \N.$}
\end{align}

To conclude the result, we now prove that $\lim_{k\to\infty}{\rm val}\eqref{Dkdr}\geq {\rm val}\eqref{P}.$
To see this, let
\begin{align*}
\mu_k:=\sup\limits_{\substack{\mu\in\R,c_{\p,\q}\geq0 }}\left\{ \mu \colon f_0(\bx)-\sum_{(\p,\q)\in N_k}c_{\p,\q}h_{\p,\q}(\bx)-\mu=0  \ \forall \bx\in\R^n\right\}.
\end{align*}
Note that $\mu_k\leq \mu_{k+1}$ for each $k\in\N.$
Since
\begin{align*}
0\in\left\{ \sigma+\sum_{j=1}^n\sigma_j \colon \sigma\in\Sigma[\bx]_{r}, \sigma_j\in\Sigma[x_j]_{d_j}, \ j=1,\ldots,n\right\},
\end{align*}
we see that for every $k\in\N,$ $\mu_k\leq {\rm val}\eqref{Dkdr}.$
Let $\epsilon>0.$
Then $f_0(\bx)-{\rm val}\eqref{P}+\epsilon>0$ for all $\bx\in {\rm\bf{K}}.$
Then, by Krivine--Stengle's positivstellensatz (Lemma~\ref{KS_positivity}), there exist $k_0\in\N$ and $c_{\p,\q}\geq0,$ $(\p,\q)\in N_{k_0}$ such that
\begin{align*}
f_0(\bx) - \sum_{(\p,\q)\in N_{k_0}}c_{\p,\q}h_{\p,\q}(\bx)-({\rm val}\eqref{P}-\epsilon)=0
\end{align*}
for all $\bx\in \R^n.$
This shows that for all $k\geq k_0,$ ${\rm val}\eqref{Dkdr}\geq\mu_k\geq\mu_{k_0}\geq {\rm val}\eqref{P}-\epsilon,$
and so,
\begin{align*}
\lim_{k\to\infty}{\rm val}\eqref{Dkdr}\geq \lim_{k\to\infty}\mu_k\geq {\rm val}\eqref{P}-\epsilon.
\end{align*}
Since $\epsilon$ is arbitrary, by \eqref{thm1rel1}, the desired results follow.
\end{proof}

\begin{proof}[Proof of Theorem~\ref{thm2}]
Assume that the condition \eqref{rel1_thm2} holds.
Then there exists $\overline \bx\in\R^n$ such that
\begin{align*}
&\sum_{\ba\in\N^n_{2\dd}}\overline y_\ba ({\lceil x_j\rceil_{\bar s}}{\lceil x_j\rceil_{\bar s}}^T)_\ba={\lceil \overline x_j\rceil_{\bar s}}{\lceil \overline x_j\rceil_{\bar s}}^T, \ j=1,\ldots,n,\\
&\sum_{\ba\in\N^n_{2\dd}}\overline y_\ba ({\lceil \bx\rceil_{\bar r}}{\lceil \bx\rceil_{\bar r}}^T)_\ba={\lceil \overline \bx\rceil_{\bar r}}{\lceil \overline \bx\rceil_{\bar r}}^T.]
\end{align*}
From this, we deduce that $\overline y_\ba=(\overline \bx)^\ba$ for all $\ba\in\N^n_{2l}$ and $\overline y_{\ba_j}=(\overline x_j)^{\alpha_j}$ for all $\alpha_j\in\N_{2{\bar s}}.$
Here, $\ba_j \in \N_{2{\bar s}}^n$ is a multi-index such that the $j$-th component is $\alpha_j$ and all other components are equal to $0.$

Since $\overline \by$ is a feasible solution of the problem \eqref{LDkdr}, we have
\begin{equation}\label{rel2_thm2}
\sum_{\ba\in\N^n_{2\dd}}(h_{\p,\q})_\ba \overline y_\ba\geq0, \  (\p,\q)\in N_k.
\end{equation}
Note that, for each $i=0,1,\ldots,m,$ $f_i(\bx)=s_i(\bx)+l_i(\bx),$ $2{\bar s}\geq \deg(s_i),$ and $2{\bar l}\geq \deg (l_i).$
Note also that we can set certain $\dd \geq {\bar s}.$ 
Now, let $i\in\{1,\ldots,m\}$ be any fixed and choose $(\p, \q) \in N_{k}$ such that $p_j = 1$ if $j = i,$ otherwise $p_j = 0,$ and $\q = \bze.$
Then, from \eqref{rel2_thm2}, we have
\begin{align*}
0\leq \sum_{\ba\in\N^n_{2\dd}}(f_i)_\ba \overline y_\ba=\sum_{\ba\in\N^n_{2{\bar s}}}(s_i)_\ba \overline \bx^\ba+\sum_{\ba\in\N^n_{2{\bar l}}}(l_i)_\ba \overline \bx^\ba=s_i(\overline\bx)+l_i(\overline\bx)=f_i(\overline\bx).
\end{align*}
Similarly, if we choose $(\p, \q) \in N_{k}$ such that $\p = \bze$ and  $q_j=1$ if $j = i,$ otherwise $p_j = 0,$ then, from \eqref{rel2_thm2}, we have
\begin{align*}
0\leq 1-\sum_{\ba\in\N^n_{2\dd}}(f_i)_\ba \overline y_\ba=1-f_i(\overline\bx).
\end{align*}
Since $i\in\{1,\ldots,m\}$ is arbitrary, $\overline \bx$ is a feasible solution of the problem \eqref{P}.
Hence we have
\begin{align*}
{\rm val}\eqref{P}\geq {\rm val}\eqref{LDkdr}=\sum_{\ba\in\N^n_{2\dd}}(f_0)_\ba \overline y_\ba=f_0(\overline \bx)\geq{\rm val}\eqref{P},
\end{align*}
where the first inequality follows from the proof of Theorem~\ref{thm1}.
Thus $\overline \bx=(\overline y_\ba)_{|\ba|=1}\in\R^n$ is an optimal solution to the problem \eqref{P}.
\end{proof}

\end{document}